\documentclass{article}
\usepackage{braket,amsfonts}
\usepackage{mathtools}
\usepackage{array}
\usepackage{enumitem}
\usepackage{amsthm}
\usepackage{algorithm2e}
\usepackage[left=2cm,right=2cm]{geometry}
\usepackage[caption=false]{subfig}
\usepackage{pgfplots}

\usepackage{algorithmic}

\usepackage{graphicx,epstopdf}

\usepackage{amsopn}

\usepackage{xspace}
\usepackage[colorlinks=True]{hyperref}
\usepackage{cleveref}
\usepackage{bold-extra}
\usepackage[most]{tcolorbox}

\colorlet{texcscolor}{blue!50!black}
\colorlet{texemcolor}{red!70!black}
\colorlet{texpreamble}{red!70!black}
\colorlet{codebackground}{black!25!white!25}

\usepackage{dsfont}

\newcommand{\R}{\mathds{R}}
\def\X{\rm{X}}

\def\Vad{{\rm V}^{\rm ad}} 
\def\VadStrict{\mathring{\rm{V}}^{\rm{ad}}}
\DeclareMathOperator*{\esssup}{ess\,sup}

\newcommand{\bu}{\bar{u}}
\newcommand{\bx}{\bar{x}}
\newcommand{\bp}{\bar{p}}
\newcommand{\ue}{u_\epsilon}
\newcommand{\ze}{z_\epsilon}
\newcommand{\zd}{z_\delta}

\newcommand{\vd}{v_\delta}

\newcommand{\xdelta}{x_\delta}
\newcommand{\bue}{\bar{u}_\epsilon}
\newcommand{\bxe}{\bar{x}_\epsilon}
\newcommand{\bxen}{\bar{x}_{\epsilon_n}}
\newcommand{\bxenk}{\bar{x}_{\epsilon_{n_k}}}
\newcommand{\bpe}{\bar{p}_\epsilon}
\newcommand{\bpen}{\bar{p}_{\epsilon_n}}
\newcommand{\bpenk}{\bar{p}_{\epsilon_{n_k}}}
\newcommand{\buen}{\bar{u}_{\epsilon_n}}
\newcommand{\buenk}{\bar{u}_{\epsilon_{n_k}}}

\newcommand{\pbe}{\bar{p}_\epsilon}

\newcommand\xs[1]{x^{1}}
\newcommand\supp{S}

\newcommand\norm[1]{\left\Vert#1\right\Vert}

\newcommand\const{\textrm{\textnormal{const}}}

\newcommand\st{\textrm{ s.t. }}

\def\xdt{{\rm{d}}t}
\def\xd{{\rm{d}}}
\def\xdn#1{{\rm{d}}#1}

\def\xBV{{\rm BV}}
\def\U{\textnormal{U}}

\def\xCzero{{\rm{C}}^{0}}
\def\xCone{{\rm{C}}^{1}} 
\def\xCtwo{{\rm{C}}^{2}} 
 
\def\xCn#1{{\rm{C}}^#1}

\def\xWn#1{{\rm{W}}^#1}

\def\xLone{{\rm{L}}^{1}}
\def\xLtwo{{\rm{L}}^{2}} 
\def\xLinfty{{\rm{L}}^{\infty}} 
\def\xLn#1{{\rm{L}}^#1}
\def\xBV{{\rm BV}}
\def\U{\textnormal{U}}

\newtheorem{ssmptn}{Assumption}
\newtheorem{rmrk}{Remark}
\newtheorem{thrm}{Theorem}
\newtheorem{prpstn}{Proposition}
\newtheorem{dfntn}{Definition}
\newtheorem{lmm}{Lemma}
\newtheorem{crllr}{Corollary}

\title{Interior point methods in optimal control problems of affine systems: Convergence results and solving algorithms}
\date{}
\author{Paul Malisani\thanks{IFP Energies nouvelles, Applied Mathematics Department, 1 et 4 avenue de Bois-Pr\'eau, 92852 Rueil-Malmaison, France (paul.malisani@ifpen.fr).}}

\begin{document}
\maketitle

\begin{abstract}
This paper presents an interior point method for pure-state and mixed-constrained optimal control problems for dynamics, mixed constraints, and cost function all affine in the control variable. This method relies on resolving a sequence of two-point boundary value problems of differential and algebraic equations. This paper establishes a convergence result for primal and dual variables of the optimal control problem. A primal and a primal-dual solving algorithm are presented, and a challenging numerical example is treated for illustration. Article accepted for publication at SIAM SICON.
\end{abstract}

\section{Introduction}
This paper deals with Optimal Control Problems (OCPs) with pure state and mixed constraints. These arise naturally in numerous engineering problems such as aerospace \cite{Betts,BrysonHo}, control of hybrid electric vehicles \cite{hev}, or innate immune response \cite{hiv}, among other examples. Unfortunately, these problems are difficult to solve \cite{Hartl,BonnansHermant,Maurer}. This difficulty mainly stems from the pure state constraints. Indeed, as shown in \cite{Hartl,Maurer}, the first-order optimality conditions of these problems imply that the adjoint state of Pontryagin can be discontinuous when state constraints switch from active to inactive and vice-versa. To handle these problems, three main approaches are found in the literature. The first is a discretization-based approach that treats an OCP as a finite-dimensional optimization problem \cite{Betts}. This approach is, in practice, the most widely used and the easiest to implement. However, these methods can be computationally slow and might lack precision. The second approach consists of computing the optimal trajectory without constraints and, step-by-step, computing the trajectory's structure, that is to say, the sequence of constrained and unconstrained arcs along the trajectory. Therefore, these methods assume that the optimal trajectory contains finitely many constrained arcs. Unfortunately, this is not always the case, and trajectory structures can be much more complicated, even for simple cases \cite{robbins}. These methods are known as continuation or homotopy methods \cite{BonnansHermant,hermantshooting,trelat_continuation}. Finally, the third approach consists in adapting Interior Point Methods (IPMs), widely studied and successfully implemented in software for numerical optimization \cite{Nocedal,WrightPrimalDual,ipopt}, to state and input-constrained OCPs \cite{lasdon67,Bonnans_log,graichen,maliOCAM,Weiser}. This approach entails minimizing an augmented cost function defined as the sum of the original cost function and so-called penalty functions, which have a diverging asymptotic behavior in the vicinity of the constraints. IPMs then define a sequence of OCPs indexed by a sequence of decreasing positive parameters converging to zero. These parameters serve as weights for the penalty functions in the augmented cost. Each OCP of this sequence is then solved as a constraints-free problem whose solution strictly satisfies the constraints and asymptotically converges to the solution of the original problem. From a practical viewpoint, these methods are appealing since off-the-shelves OCP solvers such as \cite{bvpsolve} can be used. However, adapting IPMs to OCPs is not straightforward and has yet to be wholly performed. Indeed, to be complete, this adaptation requires proving two things. Firstly, the optimal trajectories of the penalized problem are interior, i.e., strictly satisfy the constraints. Secondly, it requires establishing the convergence of the method to a point satisfying the first-order optimality conditions \cite{Hartl,Maurer,BonnansHermant}. In other words, to prove the convergence of primal variables (namely state and control variables) and the dual variables (namely the adjoint state of Pontryagin and the inequality constraints multipliers). In \cite{lasdon67,graichen}, the interiority of solutions is left as an assumption, and convergence is proved only for state and control variables. In addition, the authors assume the uniqueness of the optimal solution and strong convexity of the optimal control problem allowing them to prove convergence of the global minimum only. In \cite{Bonnans_log}, the authors prove the interiority of solutions and the convergence of control, state, and adjoint state variables in the case of linear control constraints for dynamical systems affine with respect to the control. To establish the proof, the authors assume the uniqueness of the optimal solution and the strong convexity of the optimal control problem. In \cite{Weiser}, a primal-dual IPM for OCPs is presented, and convergence of control, state, adjoint state, and constraint multipliers is proved only for control-constrained problems using a strong Legendre-Clebsch type condition, which is a sufficient condition for strong convexity \cite[Theorem 5.4]{FredericBonnans.2014}. In \cite{maliOCAM}, the authors exhibit sufficient conditions on the penalty functions guaranteeing the interiority of solutions for state and control-constrained problems. However, the convergence of the method is only proved for control and state variables, again assuming the uniqueness of the optimal solution and strong convexity of the problem.\\
This paper's contribution is proving the convergence of primal variables and constraint multipliers for pure-state and mixed constraints. In addition, the proof of convergence does not rely on the optimal solution's uniqueness or the problem's strong convexity assumptions. As in \cite{Bonnans_log,graichen}, the result is established for non-linear systems and mixed constraints affine with respect to the control. To do so, we prove that using logarithmic penalties guarantees the interiority of any locally optimal solution. In addition, we prove that the derivative of the penalty functions associated with any locally optimal solutions satisfies a uniform boundedness property. Using some standard compactness argument, we can prove weak convergence of the control and strong convergence of the state. In turn, this allows us to prove weak or strong convergence (depending on the case) of the derivatives of the penalty to the constraints multipliers. Then, the strong convergence of the adjoint state stems from the convergence of the state, control, and constraint multipliers. Finally, strong convergence of state and adjoint state allows proving strong convergence of the control variable. Finally, this paper provides a primal and a primal-dual solving algorithm based on Two Point Boundary Value Problems solver.\\
The paper is organized as follows. \Cref{sec:problem_statement} contains the problem statement, the main assumptions, and the paper's main results. In \cref{sec:preliminary_results}, some preliminary technical results are recalled. In \cref{sec:interiority}, we prove both the uniform boundedness properties of the derivatives of the penalty functions and the interiority of penalized trajectories when using logarithmic penalties. In \cref{sec:convergence}, we prove that the solutions of the penalized optimal control problem converge to a solution of the classical first-order optimality conditions of constrained optimal control problems. \Cref{sec:solving_alg} presents a primal and a primal-dual solving algorithm. Finally, in \cref{sec:algoExample}, the Robbins problem \cite{robbins} is treated using the primal and the primal-dual interior point algorithm. In addition, source codes for the Robbins and the Goddard problems \cite{seywald1993goddard} are available at \url{https://ifpen-gitlab.appcollaboratif.fr/detocs/ipm_ocp}.

\paragraph{\bf{Notations}:}
We denote $\R_-$ (resp. $\R_+$) the set of non-positive (resp. non-negative) real numbers. We denote $\mathds{N}_\ast$ (resp. $\R_\ast$) the set of non-zero natural integers (resp. real numbers). Given $p\in[1,+\infty]$, we denote $\xLn{p}(A;B)$ (or $\xLn{p}$) the Lebesgue spaces of functions from $A$ to $B$ and we denote $\norm{.}_{\xLn{p}}$ the corresponding $p$-norm. In addition, we also denote $\textrm{meas}(.)$ the Lebesgue measure on $\R$. Given $p\in[1,+\infty]$, we denote $\xWn{{1,p}}(A;B)$ the Sobolev space of measurable functions from $A$ to $B$ with weak derivative in $\xLn{p}(A;B)$. Given $n\in[0,+\infty]$, we denote $\xCn{n}(A;B)$ (or $\xCn{n}$) the set of $n$-times continuously differentiable functions from $A$ to $B$. We denote $\xBV(A)$, the set of functions with bounded variations from $A$ to $\R$. We also denote $\mathcal{M}(A)$ the set of Radon measures on $A\subset\R$. The topological dual of a topological vector space $E$ is denoted $E^*$. Given a topological vector space $E$, we denote $\sigma(E,E^*)$ the weak topology on $E$ and $\sigma(E^*,E)$ the weak $\ast$ topology on $E^*$. Let $x_n,x\in E$, we denote $x_n\rightharpoonup x$ the weak convergence in $\sigma(E,E^*)$ and let $y_n,y\in E^*$, we denote $y_n\stackrel{\ast}{\rightharpoonup}y$ the weak $\ast$ convergence in $\sigma(E^*,E)$. For $x^*\in E^*$ and $x\in E$, we denote $\langle x^*,x\rangle$, the duality product. Given $f\in \xCn{{k\geq1}}(\R^n;\R)$ we denote $f'(.)$ the gradient of the function. Given $f\in \xCn{{k\geq1}}(\R^n \times \R^m;\R^p)$, we denote $f_x'(x,y):=\frac{\partial f}{\partial x}(x,y)\in \R^{p\times n}$ (resp.$f_y'(x,y):=\frac{\partial f}{\partial y}(x,y)\in \R^{p\times m}$) and we denote $f'_{i,x}:=\frac{\partial f_i}{\partial x}(x,y)$ (resp. $f'_{i,y}:=\frac{\partial f_i}{\partial y}(x,y)$). Given $f\in \xCn{{k\geq1}}(\R^n \times \R^m;\R)$, we denote $f'_{x,i}(x,y) := \left(f'_x(x,y)\right)_i$ (resp. $f'_{y,i}(x,y) := \left(f'_y(x,y)\right)_i$). We also denote $f''_{xy}(x,y):=\frac{\partial^2f}{\partial y \partial x}(x,y)$. Let $G:X\mapsto Y$ with $X,Y$ Banach spaces, we denote $DG(x)$ the derivative of the mapping $G$ at point $x\in X$. The finite dimensional euclidean norm is denoted $\norm{.}$ and the scalar (resp. matrix) product between $x,y\in \R^n$ (resp. $x\in\R^{m\times n},y\in\R^n$) is denoted $x.y$. Given a set $E$, we denote $\vert E \vert$ its cardinal. We also denote $B_N(x,r)$ the closed ball of radius $r$ centered in $x$ for the topology induced by norm $N$. We denote $x[u,x^0]$ (or $x[u]$ if $x^0$ is fixed) the solution of the differential equations $\dot{x}=f(x,u)$ with initial condition $x^0$. Finally, we denote $\const(.)$ a positive finite constant depending on the parameters in argument.

\section{Problem statement and main result}
\label{sec:problem_statement}
\subsection{Optimal control problem}
The problem we are interested in consists of finding a solution  $(x,u)$ of the following Constrained Optimal Control Problem (COCP)
\begin{subequations}
\label{eq:all_orig_problem}
\begin{align}
\min_{(u,x)\in\U\times \X} J(x,u) &:=\varphi(x(T)) + \int_0^T \ell_1(x(t)) + \ell_2(x(t)).u(t)\xdt \\
&:= \varphi(x(T)) + \int_0^T \ell(x(t), u(t))\xdt 
\label{eq:cost_orig}\\
    \dot{x}(t) & = f_1(x(t)) + f_2(x(t)).u(t) := f(x(t),u(t))  \label{eq:def_dynamics}\\
     0 & =h(x(0),x(T)) \label{eq:def_initial_condition}\\
    0 &\geq g(x(t));\; \forall t\label{eq:def_state_const}\\
    0 &\geq a(x(t)).u(t) + b(x(t)) := c(x(t),u(t))\;a.e. \label{eq:def_mixed_const}\\
    \U& := \xLinfty([0,T];\R^m)\label{eq:def_control_space}\\
    \X& := \xWn{{1,\infty}}([0,T];\R^n)\label{eq:def_state_space}
\end{align}
\end{subequations}
where the time horizon $T>0$ is fixed.
\begin{dfntn}
    \label{def:Vad}
    We denote $\Vad \subset \U \times \R^n$ the set of admissible controls and initial conditions as follows
    \begin{equation}
    \label{eq:def_Vad}
        \Vad := \left\{(u,x^0)\in \U\times \R^n \textrm{ s.t. } \cref{eq:def_dynamics,eq:def_initial_condition,eq:def_mixed_const,eq:def_state_const} \textrm{ holds} \right\}
    \end{equation}
    The set $\Vad$ is endowed with the following norm
    \begin{equation}
        \label{eq:def_Vad_norm}
        \norm{(u,x^0)}_{\Vad} := \norm{u}_{\xLone} + \norm{x^0}
    \end{equation}
    And we denote $\VadStrict(n)$ the following set
    \begin{equation}
        \label{eq:def_VadStrict}
        \VadStrict(n) := \left\{(u, x^0) \in \Vad \textrm{ s.t. } \begin{cases}
            g(x[u, x^0](t))&< 0,\; \forall t\\
            \esssup_t c(x[u, x^0], u)&\leq -\frac{1}{n}
        \end{cases}\right\}
    \end{equation}
\end{dfntn}

\subsection{Main assumptions and technical definitions}
\begin{ssmptn}
\label{ass:C1_data}
The functions $\ell:\R^n\times\R^m\mapsto\R$, $f:\R^n\times\R^m\mapsto\R^n$, $g:\R^n\mapsto\R^{n_g}$, $c:\R^n \times \R^m \mapsto \R^{n_c}$ are at least twice continuously differentiable.
\end{ssmptn}

\begin{ssmptn}
    \label{ass:interior_accessibility}
    Any locally optimal solution $(x[\bu, \bx^0],\bu)$ such that $(\bu, \bx^0)\in\Vad$ satisfies the following interiority accessibility assumption
    \begin{equation}
        \label{eq:def_interior_accessibility}
        (\bu, \bx^0) \in \textrm{\rm{cl}}_{\norm{.}_{\Vad}}\left\{\liminf_{n} \VadStrict(n) \right\}:=\Vad_\infty
    \end{equation}
    where $\textrm{\rm{cl}}_{\norm{.}_{\Vad}}$ stands for the closure in the $\norm{.}_{\Vad}$-topology. 
\end{ssmptn}

\begin{ssmptn}
\label{ass:bounded_if_state_set}
The set of admissible initial-final states $h^{-1}(\lbrace 0\rbrace)\subset \R^n \times \R^n$ from \cref{eq:def_initial_condition} is closed and bounded. 
\end{ssmptn}

\begin{ssmptn}
    \label{ass:bounded_control_set}
    There exists $R_v \in (0,+\infty)$ such that 
    \begin{equation}
    \norm{(u,x^0)}_{\Vad}\leq R_v,\;\forall (u,x^0)\in \Vad
    \end{equation}
    and for all $R_v\in(0,+\infty)$, there exists $R_x\in (0,+\infty)$ such that
    \begin{equation}
    \norm{x[u,x^0]}_{\xLinfty} \leq R_x,\; \forall \norm{(u,x^0)}_{\Vad} \leq R_v
    \end{equation}
\end{ssmptn}

\begin{dfntn}[Set of near state-saturated times and near-saturated indices]
For all $(u,x^0)\in \Vad$ from \cref{def:Vad} and $\forall \delta \geq0$ we define the set of near state-saturated times (resp. mixed-saturated times), denoted $\supp^g_{u,x^0}$ (resp. $\supp^c_{u,x^0}$) , as follows
\begin{align}
    \supp^g_{u,x^0}(\delta) &:= \left\lbrace t\in [0,T] \st \max_i g_i(x[u, x^0](t))\geq -\delta\right\rbrace\label{eq:def_supp_g}\\
    \supp^c_{u,x^0}(n) &:= \left\lbrace t\in [0,T] \st \max_i c_i(x[u, x^0](t), u(t))\geq -\frac{1}{n}\right\rbrace\label{eq:def_supp_c}
\end{align}
In addition, we define the set of near state-saturated indices (resp. mixed-saturated indices), denoted $I^g_{u,x^0}$ (resp. $I^c_{u,x^0}$) , as follows
\begin{align}
    I^g_{u,x^0}(t,\delta) &:= \left\lbrace i\in\{1,\dots,n_g\}\st g_i(x[u, x^0](t))\geq -\delta\right\rbrace\label{eq:def_indices_g}\\
    I^c_{u,x^0}(t,n) &:= \left\lbrace i\in\{1,\dots,n_c\} \st c_i(x[u, x^0](t), u(t))\geq -\frac{1}{n} \right\rbrace\label{eq:def_indices_c}
\end{align}
\end{dfntn}

\begin{ssmptn}
    \label{ass:qualification_condition_mixed_constraint}
    For all $(u,x^0) \in \Vad$, the mixed constraints \cref{eq:def_mixed_const} satisfy the following qualification condition. There exists $\gamma>0$ and $n \in \mathds{N}_\ast$ such that
    \begin{equation}
        \gamma \norm{\xi} \leq \norm{c'_{I^c_{u,x^0}(t,n),u}(x[u,x^0](t), u(t))^\top.\xi},\;\forall \xi \in \R^{\vert I^c_{u,x^0}(t,n)\vert},\;\textrm{a.a. } t\in [0,T]
    \end{equation}
\end{ssmptn}

\begin{ssmptn}
    \label{ass:robinson_qualification}
    The set of singular multipliers for problem \cref{eq:all_orig_problem} is empty.
\end{ssmptn}

\begin{rmrk}
Sufficient conditions on pure-state and mixed constraints such that \cref{ass:robinson_qualification} holds are given in \cite{BonnansHermant}. In addition, these assumptions guarantee the existence of optimal solutions of problem \cref{eq:all_orig_problem} (see \cite{haberkorn}). 
\end{rmrk}

\begin{dfntn}[State-constraint measure]
\label{def:mui}
For all $(u, x^0) \in B_{\xLinfty}(0,R_u) \times B_{\norm{.}}(0,R_x)$ and for all $E \subset \R $, we denote $m[u, x^0, g_i]$ the push-forward $g_i$-measure of $E$ defined as follows
\begin{equation}
    \label{eq:def_mui_alpha1_alpha2}
    m[u, x^0,g_i](E) := \textrm{meas}\left(\left(g_i\circ x[u, x^0]\right)^{-1}(E)\right)
\end{equation}
\end{dfntn}

\subsection{First-order necessary conditions of stationarity} In this section, the first order necessary conditions of optimality for Problem \cref{eq:all_orig_problem}. To do so, let us introduce the infamous pre-Hamiltonian.
\begin{dfntn}[pre-Hamiltonian]
\label{def:pre_hamiltonian}
The pre-Hamiltonian $H:\R^n \times \R^m \times \R^n\mapsto \R$ of the optimal control problem is defined by
\begin{equation}
    \label{eq:def_pre_hamiltonian}
    H(x,u,p) := \ell(x,u) + p.f(x,u)
\end{equation}
\end{dfntn}

\begin{dfntn}[Stationary point]
\label{def:pontryagin_extremal}
The trajectory $(\bx, \bu)$ with associated multipliers $(\bp, \bar{\mu}, \bar{\nu},\bar{\lambda})\in \xBV([0,T];\R)^n \times \mathcal{M}([0,T])^{n_g} \times  \xLinfty([0,T];\R^{n_c}) \times \R^{n_h}$, is a stationary point for Problem \eqref{eq:all_orig_problem} if it satisfies
\begin{subequations}
\label{eq:all_first_order}
\begin{align}
    \dot{\bar{x}}(t) =& f(\bar{x}(t),\bu(t))\label{eq:first_order_1}\\
    -\xdn{\bar{p}}(t) =& \left[H_x'(\bx(t), \bu(t),\bp(t)) + \sum_{i=1}^{n_c}c_{i,x}'(\bx(t), \bu(t))\bar{\nu}_i(t)\right] \xdt
+\sum_{i=1}^{n_g}g'_i(\bx(t))\xdn{\bar{\mu}}_i(t)\label{eq:first_order_2}\\
    0=& H_u'(\bx(t),\bu(t), \bp(t))+\sum_{i=1}^{n_c} c_{i,u}'(\bx(t), \bu(t)) \bar{\nu}_i(t)\label{eq:first_order_3}\\
    0=&h(\bx(0), \bx(T)) \label{eq:first_order_4}\\
    0  =&\bp(0) + h_{x(0)}'(\bx(0),\bx(T))^\top.\bar{\lambda}\label{eq:first_order_5}\\
     0=&\bp(T) - \varphi'(\bx(T)) -  h_{x(T)}'(\bx(0),\bx(T))^\top.\bar{\lambda}\label{eq:first_order_6}\\
   0=& \int_0^T g_i(\bx(t))\xdn{\bar{\mu}}_i(t),\;\;i=1,\dots,n_g\label{eq:first_order_7}\\
    0=&\int_0^T c_i(\bx(t), \bu(t))\bar{\nu}_i(t)\xdt,\;\;i=1,\dots,n_c\label{eq:first_order_8}\\
    0=&\bar{\lambda}^\top.h(\bx(0),\bx(T)) \label{eq:first_order_9}\\
    0 \leq & \xdn{\bar{\mu}}_i(t),\;\;i=1,\dots,n_g\label{eq:first_order_11}\\
    0 \leq &\bar{\nu}_i(t),\;\;i=1,\dots,n_c\label{eq:first_order_12}\\
    0 = &\bar{\mu}_i(T), \;\;,\;\;i=1,\dots,n_g\label{eq:first_order_13}
\end{align}
\end{subequations}
\end{dfntn}
It is a well-established results \cite{Maurer,BonnansHermant} that any local solution of problem \cref{eq:all_orig_problem} is a stationary point as defined in \cref{def:pontryagin_extremal}. Unfortunately, solving Problem \cref{eq:all_first_order} is a difficult task. Indeed, the dual variable $d\mu$ associated with the state constraints appearing in \cref{eq:first_order_2,eq:first_order_7,eq:first_order_11,eq:first_order_13} are Radon measures, therefore, in full generality, they can be decomposed in an absolutely continuous measure with respect to the Lebesgue measure, a discrete and finally a singular part. Computing these measures' discrete and singular parts can be dramatically complex.
\subsection{Penalized Optimal Control Problem (POCP)}
To solve problem \cref{eq:all_orig_problem}, we use an interior point method based on log-barrier functions defined as follows 
\begin{dfntn}[log-barrier function]
\label{def:log_barrier_function}
The log-barrier function $\psi:\R\mapsto\R$ is defined as follows
\begin{equation}
    \psi(x):=\begin{cases}
        -\log(-x)& \forall x<0\\
        +\infty&\textrm{otherwise}
    \end{cases}
\end{equation}
\end{dfntn}
The corresponding penalized optimal control problem is defined as follows
\begin{subequations}
\label{eq:def_log_barrier_ocp}
\begin{align}
    \min_{(x,u)\in \X \times \U} J_\epsilon(x,u) &:=J(x,u) +\epsilon\int_0^T\left[\sum_{i=1}^{n_g}\psi\circ g_i(x(t)) +\sum_{i=1}^{n_c} \psi \circ c_i(x(t), u(t))\right] \xdt\\
    \dot{x}(t) & = f(x(t),u(t))\\
    0&=h(x(0),x(T))
\end{align}
\end{subequations}

The pre-Hamiltonian associated with this penalized problem is defined here after
\begin{dfntn}[Penalized pre-Hamiltonian]
\label{def:penalized_pre_hamiltonian}
The penalized pre-Hamiltonian $H^\psi:\R^n \times \R^m \times \R^n \times \R \mapsto \R$ of POCP \cref{eq:def_log_barrier_ocp} is defined by
\begin{equation}
    \label{eq:def_penalized_pre_hamiltonian}
    H^\psi(x,u,p, \epsilon) := H(x,u,p)+ \epsilon\left(\sum_{i=1}^{n_g}\psi(g_i(x))+\sum_{i=1}^{n_c} \psi(c_i(x,u))\right)
\end{equation}
\end{dfntn}

\begin{dfntn}[Penalized stationary point]
\label{def:penalized_pontryagin_extremal}
The trajectory $(\bxe, \bue)$ with associated multipliers $(\pbe, \bar{\lambda}_\epsilon)\in \xWn{{1,1}}([0,T];\R^n) \times \R^{n_h}$, is a penalized stationary point for Problem \eqref{eq:def_log_barrier_ocp} if it satisfies
\begin{subequations}
\label{eq:penalized_pontryagin_extremal}
\begin{align}
    \dot{\bar{x}}_{\epsilon}(t) =& f(\bxe(t),\bue(t))\label{eq:penalized_pontryagin_extremal_1}\\
    \dot{\bp}_\epsilon(t) =& - {H_x^{\psi}}'(\bxe(t),\bue(t),\pbe(t), \epsilon)\label{eq:penalized_pontryagin_extremal_2}\\
    0=&{H_u^{\psi}}'(\bxe(t),\bue(t),\pbe(t), \epsilon)\label{eq:penalized_pontryagin_extremal_3}\\
    0=&h(\bxe(0), \bxe(T))\label{eq:penalized_pontryagin_extremal_4}\\
    0 =& \bp_\epsilon(0) + h'_{x(0)}(\bxe(0), \bxe(T))^\top.\bar{\lambda}_{\epsilon}\label{eq:penalized_pontryagin_extremal_5}\\
    0 =& \bp_\epsilon(T) -\varphi'(\bxe(T))-  h'_{x(T)}(\bxe(0), \bxe(T))^\top.\bar{\lambda}_{\epsilon}\label{eq:penalized_pontryagin_extremal_6}
\end{align}
\end{subequations}
\end{dfntn}

\begin{rmrk}
 Handling mixed constraints using an interior point method rather than a Pontryagin minimization has two main advantages. Firstly, it allows for using off-the-shelves index-1 BVPDAE solvers. Secondly, it also allows for implementing primal-dual methods in optimal control as described in \cite{Weiser}, which are numerically efficient as will be illustrated later in \cref{sec:algoExample} and in the numerical examples available at \url{https://ifpen-gitlab.appcollaboratif.fr/detocs/ipm_ocp}.
\end{rmrk}
\subsection{Contribution of the paper}
The main contribution of the paper is a convergence theorem for interior point methods optimal control problem with logarithmic penalty functions.
\begin{thrm}
\label{thm:first_order_convergence}
Let $(\epsilon_n)$ be a sequence of decreasing positive parameters with $\epsilon_n\rightarrow0$. The associated sequence of penalized stationary points $(\bx_{\epsilon_n}, \bu_{\epsilon_n}, \bar{p}_{\epsilon_n}, \bar{\lambda}_{\epsilon_n})_n$ as defined in \cref{def:penalized_pontryagin_extremal} contains a subsequence converging to a stationary point $(\bx, \bu,\bp, \bar{\mu}, \bar{\nu},\bar{\lambda})$ of the original problem as defined in \cref{def:pontryagin_extremal} as follows
\begin{subequations}
\begin{align}
     \norm{\buenk - \bu}_{\xLone} & \rightarrow 0 \label{eq:conv_u_L2}\\
      \norm{\bxenk -\bx}_{\xLinfty}&\rightarrow 0\label{eq:conv_x_Linfty}\\
      \norm{h\left(\bxenk(0), \bxenk(T)\right)-h\left(\bx(0), \bx(T)\right)}& \rightarrow 0\label{eq:conv_if}\\
      \left\vert J_{\epsilon_{n_k}}(\bxenk, \buenk) - J(\bx, \bu)\right\vert & \rightarrow 0 \label{eq:conv_cost}\\
      \norm{\bar{\lambda}_{\epsilon_{n_k}} - \bar{\lambda}}&\rightarrow 0 \label{eq:conv_lambda}\\
  \norm{\bpenk - \bar{p}}_{\xLone} &\rightarrow 0\label{eq:conv_p_L1}\\
\textcolor{black}{\epsilon_{n_k}\psi'\circ c_i(\bxenk, \buenk)}&\textcolor{black}{ \stackrel{\ast}{\rightharpoonup} \bar{\nu}_i ,\;\;\;i=1,\dots,n_c}\label{eq:conv_weak_nu}\\
   \epsilon_{n_k}\psi'\circ g_i(\bxenk)\xdt&\stackrel{\ast}{\rightharpoonup} \xdn{\bar{\mu}}_i ,\;\;\;i=1,\dots,n_g\label{eq:weak_star_conv_pen_state}
\end{align}
\end{subequations}
\end{thrm}

\section{Preliminary results}
\label{sec:preliminary_results}
This section gathers useful definitions and preliminary results, which will be recurrently used throughout the paper. The proofs of these results are given in \cref{sec:proof_annex}.

\begin{prpstn}[State Lipschitz continuity]
\label{prop:major_Linf_L1}
Let $u_1,u_2\in\U$ and $x_1^0, x_2^0 \in \R^n$, there exists $\const(f) < +\infty $ such that
\begin{equation}
    \parallel x[u_1, x_1^0]-x[u_2, x_2^0]\parallel_{\xLinfty} \leq \const(f)(\parallel u_1 - u_2\parallel_{\xLone} + \norm{x_1^0 - x_2^0})
\end{equation}
\end{prpstn}

\begin{proof}
See \cref{sec:proof_lipschitz}
\end{proof}

\begin{prpstn}
\label{lemme:weak_strong_l2_conv}
    Let $(u_n, x^0_n)_n$ converging to $(\bu, \bx_0)$ in the weak topology $\sigma(\xLtwo\times \R^n, \xLtwo \times \R^n)$ then we have
    \begin{equation}
        \norm{x[u_n, x^0_n] - x[\bu, \bx^0]}_{\xLinfty} \rightarrow 0
    \end{equation}
    and for all $\alpha \in \xCzero(\R^n;\R^{p \times m})$, $\beta \in \xCzero(\R^n;\R^p)$ we have 
    \begin{equation}
        \alpha(x[u_n, x_n^0]).u_n + \beta(x[u_n, x_n^0]) \stackrel{\ast}{\rightharpoonup} \alpha(x[\bu, \bx^0]).\bu + \beta(x[\bu, \bx^0])\label{eq:weak_ast_conv_mixed_const}
    \end{equation}
\end{prpstn}

\begin{proof}
    See \cref{app:weak_strong_convergence}
\end{proof}

\begin{rmrk}
The affine property of the dynamics and the mixed constraints \cref{eq:def_dynamics,eq:def_mixed_const} are crucial in the proof of \cref{lemme:weak_strong_l2_conv}. Indeed, if the dynamics, for example, is not affine with respect to the control, the sequence $(x[u_n])_n$  uniformly converges to a limit which can be different from $x[\bu]$. Take, for example, $\dot{x}(t) = u(t)^2$ and $u_n:= \sin(nt)$. Then $u_n\rightharpoonup \bu=0$ and $x_n(t)\rightarrow t/2\neq x[\bu](t)$.
\end{rmrk}

\begin{prpstn}\label{prop:lower_bound_measure}
For all $u\in\U$ satisfying \cref{ass:bounded_control_set}, for all $x^0\in \R^n$ bounded, let $E \subseteq g_i\circ x[u, x^0]([0,T])\subset\R$ be a Lebesgue-measurable set, the state-constraint measure from \cref{def:mui} is lower bounded as follows
\begin{equation}
\label{eq:lower_bound_measure}
   m[u, x^0 ,g_i](E)\geq \const(f,g)\textrm{meas}(E)
\end{equation}
\end{prpstn}

\begin{proof}
See \cref{sec:proof_measure}

\end{proof}

\begin{prpstn}
\label{prop:maximal_distance}
For all $\delta>0$, $\exists (G_\delta, C_\delta)>0$ such that, $\forall (u, x^0) \in \Vad_\infty$, $\exists (v,y^0)\in B_{\norm{.}_{\Vad}}((u, x^0),\delta)\cap\Vad$ satisfying the following condition

\begin{align}
    \sup_t g_i(x[v, y^0](t))&\leq -2G_\delta,\;i=1,\dots,n_g\\
    \esssup_t c_i(x[v, y^0](t), v(t))&\leq -2C_\delta,\;i=1,\dots,n_c
\end{align}
and we also have
\begin{align}
    g_i(x[v, y^0](t))\leq g_i(x[u, x^0](t))-G_\delta,\;\forall t\in S^g_{u,x^0}(G_\delta),\;i=1,\dots,n_g\\ 
    c_i(x[v, y^0](t), v(t))\leq c_i(x[u, x^0](t), u(t))-C_\delta,\;\textrm{a.a. } t\in S^c_{u,x^0}(C_\delta),\;i=1,\dots,n_c
\end{align}
     
\end{prpstn}
\begin{proof}
See \cref{sec:proof_maximal_distance}.
\end{proof}

\section{Uniform boundedness and interiority analysis of penalized optimal solutions}
\label{sec:interiority}
\subsection{State constraints analysis}
In the following, we present an interior point optimal control problem to handle pure state constraints, which writes:
\begin{subequations}
\label{eq:all_state_penalized_problem}
    \begin{align}
          \min_{(x, u)} J^1_\epsilon(x, u) =&\varphi(x(T)) + \int_0^T\left[\ell(x(t),u(t))+\epsilon\sum_{i=1}^{n_g} \psi(g_i(x(t))\right]\xdt\label{eq:def_state_pen_problem_1}\\
          \dot{x}(t) & = f(x(t),u(t))\label{eq:def_state_pen_problem_2}\\
          h(x(0), x(T)) & = 0 \label{eq:def_state_pen_problem_3}\\
          c(x(t), u(t))& \leq 0\label{eq:def_state_pen_problem_4}
    \end{align}
\end{subequations}

\begin{lmm}
\label{thm:interior_state}
$\forall \epsilon >0$, any locally optimal solution $(x[\ue, x^0_\epsilon], \ue)$ of Problem \cref{eq:all_state_penalized_problem} satisfies
\begin{equation}
    g_i(x[\ue, x^0_\epsilon](t)) < 0,\forall t\in[0,T],\;\;i=1,\dots,n_g\label{eq:interiority_state}
\end{equation}
and $\exists K_g<+\infty$ such that $\forall \epsilon \in(0,\epsilon_0)$ we have
\begin{equation}
    \norm{\epsilon \psi'\circ g_i(x[\ue, x^0_\epsilon])}_{\xLone}\leq K_g,i=1,\dots,n_g\label{eq:boundedness_lambda}
\end{equation}
\end{lmm}
\begin{proof}
It is sufficient to prove for $n_g=1$, that is to say, for just one state constraint. Assume $(\ue, x^0_\epsilon) \in \Vad_\infty$ is a local optimal solution satisfying 
\begin{equation}
\label{eq:ass_touchpoint_g}
    \sup_t g(x[\ue, x_\epsilon^0](t)) = 0
\end{equation}
From \cref{prop:maximal_distance}, $\forall \delta >0, \exists (\vd, x_\delta^0)\in B_{\norm{.}_{\Vad}}((\ue, x^0_\epsilon),\delta) \cap \VadStrict$ and $G_\delta >0$ such that
\begin{subequations}
    \begin{align}
        g(x[\vd, x_\delta^0](t))&\leq -2 G_\delta,\; \forall t \in[0,T]\\
         g(x[\vd, x_\delta^0](t))&\leq g(x[\ue, x^0_\epsilon](t))-G_\delta,\; \forall t\in \supp^g_{\ue, x^0_\epsilon}(G_\delta) \label{eq:inequality_sur_suppbue}
    \end{align}
\end{subequations}
with $\supp^g_{\ue, x^0_\epsilon}(G_\delta)\neq \emptyset$. In the following, to alleviate the notations, we denote 
\begin{align}
    \zd &:= (\vd, x_\delta^0)\\
    \ze &:= (\ue, x^0_\epsilon)\\
    \Delta z &:= \zd - \ze\\
    \Delta x &:= x[\zd] - x[\ze]\\
    \Delta g &:= g(x[\zd]) - g(x[\ze])
\end{align}
Now, one can exhibit an upper-bound on the difference $J^1_\epsilon(x[\zd], \vd)-J^1_\epsilon(x[\ze], \ue)$ as follows
\begin{equation}
    J^1_\epsilon(x[\zd], \vd)-J^1_\epsilon(x[\ze], \ue) = \Delta_1 + \epsilon \Delta_2\label{eq:def_split_g_epsilon}
\end{equation}
where
\begin{align}
    \Delta_1 &:= 
        \varphi(x[\zd](T))-\varphi(x[\ze](T)) 
        + \int_0^T\left[\ell(x[\zd](t),\vd(t))-\ell(x[\ze](t),\ue(t))\right]\xdt\label{eq:def_delta_1}\\
     \Delta_2 &:= \int_0^T\left[\psi\circ g(x[\zd](t))-\psi \circ g(x[\ze](t))\right]\xdt\label{eq:def_delta_2}
\end{align}
Now, let us upper-bound $\Delta_1$
\begin{align}
    \Delta_1 \leq&
    \int_0^T\const(\ell)\left(\norm{x[\zd](t)-x[\ze](t)}+\norm{\vd(t)-\ue(t)}\right)\xdt
    +\const(\varphi)\norm{\Delta x}_{\xLinfty}\\
    \leq&\const(\ell, T, \varphi)\norm{\Delta x}_{\xLinfty}+\const(\ell)\norm{\vd-\ue}_{\xLone} \leq\const(\ell,f,g, \varphi, T, R_v, R_x)\label{eq:upperbound_delta1}
\end{align}
Now, let us upper-bound $\Delta_2$ from \cref{eq:def_delta_2}. To do so, let us introduce the following useful subsets of $[0;T]$
 \begin{align}
    E_1&:=(g\circ x[\ze])^{-1}\left((-\infty,-G_\delta]\right)\label{eq:def_set_a1}\\
    E_2(\rho)&:=(g\circ x[\ze])^{-1}\left((-G_\delta,-\rho]\right)\label{eq:def_set_a2}\\
    E_3(\rho)&:=(g\circ x[\ze])^{-1}\left((-G_\delta,-\rho)\right)\label{eq:def_set_a3}
\end{align}
Given \cref{eq:ass_touchpoint_g}, for all $\rho\in [0,G_\delta)$, these sets are not empty and $\forall t \notin \left(E_1 \cup E_2(\rho)\right)$ we have $\psi\circ g(x[\zd](t))- \psi\circ g(x[\ze](t))<0$ which yields
\begin{equation}
    \Delta_2 \leq \int_{E_1} \psi \circ g(x[\zd](t)) - \psi \circ g(x[\ze](t))\xdt
    + \int_{E_2(\rho)} \psi\circ g(x[\zd](t))- \psi \circ g(x[\ze](t))\xdt
\end{equation}
By convexity of the $\log$ penalty, i.e. $\psi$, we have
\begin{equation}
    \label{eq:majoS1}
    \int_{E_1} \psi\circ g(x[\zd](t)) - \psi\circ g(x[\ze](t))\xdt\leq
    \int_{E_1} \psi'(G_\delta) \norm{\Delta g}_{\xLinfty} \xdt:=\const(T,f,g,G_\delta)
\end{equation}
In addition, $\forall t \in E_2(\rho)$, we have
\begin{equation}
     \int_{E_2(\rho)}\psi\circ g(x[\zd](t)) - \psi\circ g(x[\ze](t))\xdt =
      \int_{E_2(\rho)} \left(\int_0^1 \psi'(g(x[\ze](t) + s \Delta g(t)) \xd s\right)\Delta g(t) \xdt
\end{equation}
Since $\forall t\in E_2(\rho), \,\Delta g(t)  <  -G_\delta$, we also have 
\begin{equation}
    \label{eq:majoS2}
    \int_{E_2(\rho)}\psi\circ g(x[\zd](t)) - \psi\circ g(x[\ze](t))\xdt \leq
    -G_\delta\int_{E_2(\rho)}\left(\int_{0}^{1} \psi'(g(x[\ze](t)) + s\Delta g(t))\xd s\right) \xdt
\end{equation}
From the mean value theorem, $\forall t\in E_2(\rho)$, $\exists \sigma_t$ such that
\begin{equation}
\label{eq:mean_value_g}
    \psi'(g(x[\ze](t)) + \sigma_t \Delta g(t))= \int_{0}^{1} \psi'(g(x[\ze](t)) + s\Delta g(t))\xd s
\end{equation}
Since for all $t\in E_2(\rho)$, we have $g(x[\ze)(t)) - g(x[\zd](t))\geq G_\delta$ and since $\psi'$ is strictly increasing we have $\sigma_t \in (0,1)$ and 
\begin{equation}
    \psi'\circ g(x[\zd](t)) < \psi'(g(x[\ze](t)) + \sigma_t \Delta g(t)) <\psi'\circ g(x[\ze](t))
\end{equation}
From the intermediate value theorem, $\exists \bar{\sigma}\in(0,1)$ such that $\forall t\in E_2(\rho)$ we have
\begin{equation}
\label{eq:lower_bound_intermediate}
    \psi'(g(x[\ze](t)) + \sigma_t \Delta g(t)) \geq (1 - \bar{\sigma})\psi'\circ g(x[\ze](t)) + \bar{\sigma} \psi'\circ g(x[\zd](t))
\end{equation}
Gathering \cref{eq:majoS2,eq:mean_value_g,eq:lower_bound_intermediate} yields
\begin{multline}
    \int_{E_2(\rho)}\psi\circ g(x[\zd](t)) - \psi\circ g(x[\ze](t))\xdt  \\
   \leq - G_\delta \int_{E_2(\rho)}\bigg( (1-\bar{\sigma})\psi' \circ g(x[\ze](t)) + \bar{\sigma} \psi' \circ g(x[\zd](t)) \bigg)\xdt\\
    \leq- G_\delta\left( (1 - \bar{\sigma})\int_{E_2(\rho)}\psi' \circ g(x[\ze](t))\xdt + \const(\bar{\sigma}, \psi, T,g, G_\delta)\right)\label{eq:upperbound_delta2_over_E2}
\end{multline}
Gathering \cref{eq:majoS1,eq:upperbound_delta2_over_E2} we have 
\begin{equation}
    \label{eq:upperbound_delta2}
    \epsilon \Delta_2 \leq \const(T,f,g,G_\delta,\epsilon_0, \psi, \bar{\sigma}) -\epsilon G_\delta(1 - \bar{\sigma})\int_{E_2(\rho)}\psi' \circ g(x[\ze](t))\xdt
\end{equation}
Now, let us prove that any optimal solution is strictly interior with respect to the state constraint. The proof is by contradiction. Using \cref{def:mui}, one can make the following change in measure
\begin{equation}
\int_{E_3(\rho)}\psi' \circ g(x[\ze](t))\xdt = \int_{-G_\delta}^{-\rho}\psi'(s)m[\ze,g](\xd s)
\end{equation}
Then, using \cref{prop:lower_bound_measure,eq:def_set_a2,eq:def_set_a3} yields
\begin{align}
    \int_{E_2(\rho)}\psi' \circ g(x[\ze](t))\xdt&\geq\int_{E_3(\rho)}\psi' \circ g(x[\ze](t))\xdt\geq\const(f,g)\left(\psi(\rho)-\psi(G_\delta)\right)\label{eq:proof_change_variable}
\end{align}
Gathering \cref{eq:def_split_g_epsilon,eq:upperbound_delta1,eq:upperbound_delta2,eq:proof_change_variable} yields that $\forall \rho>0$ we have
\begin{equation}
    J^1_\epsilon(\zd) - J^1_\epsilon(\ze) \leq\const(\ell,f,g, \varphi, T, \epsilon_0,\psi, G_\delta, R_u, \bar{\sigma})
-\epsilon G_\delta \const(f,g, \bar{\sigma}) \left(\psi(\rho)-\psi(G_\delta)\right)
\end{equation}
For $\rho$ small enough, this yields $J^1_\epsilon(\zd)<J^1_\epsilon(\ze)$ and contradicts the local optimality of $\ze$ and proves \cref{eq:interiority_state}. Now, to prove \cref{eq:boundedness_lambda}, let us ensure that the left-hand side of this equation is well-defined. From \cref{eq:interiority_state} we have $(g(x[\ze]))^{-1}(\{0\}) =\emptyset$, thus
\begin{equation}
    [0,T]\ = \lim_{\rho \rightarrow 0}((g(x[\ze]))^{-1}((-\infty, \rho))
    \label{eq:density_S3_S2}
\end{equation}
Hence, using \cref{def:mui}, one has
\begin{equation}
    \norm{\psi'\circ g(x[\ze])}_{\xLone} =\int_0^T\psi' \circ g(x[\ze](t))\xdt:=\lim_{\rho\rightarrow 0}\int_{-\infty}^{-\rho}\psi'(s)m[\ze,g](\xd s)
\end{equation}
which is well-defined. Now, let us prove \cref{eq:boundedness_lambda} by contradiction and assume that
\begin{equation}
    \forall K_g>0, \exists \epsilon>0 \st \norm{\epsilon \psi'\circ g(x[\ze]) }_{\xLone}>K_g
\end{equation}
Then, from \cref{eq:def_set_a1,eq:def_set_a2,eq:density_S3_S2}, one has
\begin{equation}
\lim_{\rho \rightarrow 0}\epsilon\int_{-G_\delta}^{-\rho}\psi'(s)m[\ze,g](\xd s) > K_g - \epsilon \int_{-\infty}^{-G_\delta} \psi'(s) m[\ze,g](\xd s)> K_g -\frac{\epsilon_0 T}{G_\delta} \label{eq:upperbound_l1_gammaxprime}
\end{equation}
Gathering \cref{eq:upperbound_delta1,eq:upperbound_delta2,eq:upperbound_l1_gammaxprime} yields
\begin{equation}
   \Delta_1 + \epsilon \Delta_2 \leq 
    \const(\ell,f,g, \varphi,T, \epsilon_0,\psi, G_\delta, R_v, R_x, \bar{\sigma})
-G_\delta(1 - \bar{\sigma}) K_g
\end{equation}
Since $G_\delta(1 - \bar{\sigma})>0$, $\exists K_g>0$ such that $\Delta_1 + \epsilon \Delta_2<0$ which contradicts the optimality of $\ze$, proves \cref{eq:boundedness_lambda} and concludes the proof.
\end{proof}

\subsection{Mixed constraints interiority analysis}

\begin{lmm}
\label{thm:interior_mixed_const}
There exists a constant $K_c<+\infty$ such that for all $\epsilon>0$ and for any $(x[\ue, x^0_\epsilon], \ue)$ locally optimal solution of Problem \cref{eq:def_log_barrier_ocp} the following holds
\begin{align}
\norm{\epsilon \psi'\circ c_i(x[\ue, x^0_\epsilon], \ue)}_{\xLone} & \leq K_c, \;i=1,\dots,n_c \label{eq:bounded_mixed_penalty_derivative}
\end{align}
\end{lmm}

\begin{proof}
It is sufficient to prove the case where $n_c=1$, i.e., when there is a single mixed constraint. From \cref{prop:maximal_distance}, $\forall \delta >0, \exists (\vd, \xdelta^0)\in B_{\norm{.}_{\Vad}}((\ue, x^0_\epsilon),\delta) \cap \Vad$ and $C_\delta >0$ such that
\begin{subequations}
    \begin{align}
        c(x[\vd, \xdelta^0](t), \vd(t))&\leq -2 C_\delta,\; \textrm{ a.a. }t\in[0,T]\\
         c(x[\vd, \xdelta^0](t), \vd(t))&\leq c(x[\ue, x^0_\epsilon](t), \ue(t))-C_\delta,\; \forall t\in \supp^c_{\ue, x^0_\epsilon}(C_\delta) \label{eq:inequality_sur_supp_c_ue}
    \end{align}
\end{subequations}
with $\supp^c_{\ue, x^0_\epsilon}(C_\delta)\neq \emptyset$. In the following, to alleviate the notations, we denote 
\begin{align}
    \zd &:= (\vd, \xdelta^0)\\
    \ze &:= (\ue, x^0_\epsilon)\\
    \Delta z &:= \zd - \ze\\
    \Delta x &:= x[\zd] - x[\ze]\\
    \Delta g &:= g(x[\zd]) - g(x[\ze])\\
    \Delta c &: = c(x[\zd],\vd) - c(x[\ze], \ue)
\end{align}
In addition, From \cref{thm:interior_state}, and by continuity of the mapping $z\mapsto x[z]$ one can chose $\delta>0$ such that the following holds
\begin{align}
    \sup_t g(x[\zd](t))& < 0\label{eq:interiority_state_mixed}\\
    \norm{\epsilon \psi' \circ g(x[\zd])}_{\xLone} &\leq 2\norm{\epsilon \psi' \circ g(x[\ze])}_{\xLone}\leq 2K_g\label{eq:upper_bound_state_mixed}
\end{align}
Now, one can exhibit an upper-bound on the difference $J_\epsilon(x[\zd], \vd)-J_\epsilon(x[\ze], \ue)$ as follows
\begin{equation}
    J_\epsilon(x[\zd], \vd)-J_\epsilon(x[\ze], \ue) = \Delta_1  + \Delta_2 + \epsilon \Delta_3\label{eq:def_split_c_epsilon}
\end{equation}
where
\begin{align}
    \Delta_1 &:=
        \varphi(x[\zd](T))-\varphi(x[\ze](T))  + \int_0^T\left[\ell(x[\zd](t),\vd(t))-\ell(x[\ze](t),\ue(t))\right]\xdt\label{eq:def_delta_1_c}\\
    \Delta_2 & := \epsilon\int_0^T \sum_i\left[\psi \circ g_i(x[\zd](t))-\psi \circ g_i(x[\ze](t)) \right]\xdt\label{eq:def_delta_2_c}\\
     \Delta_3 &:= \int_0^T\left[\psi\circ c(x[\zd](t), \vd(t))-\psi \circ c(x[\ze](t), \ue(t))\right]\xdt\label{eq:def_delta_3_c}
\end{align}
Now, let us upper-bound $\Delta_1$
\begin{multline}
    \Delta_1 \leq 
    \int_0^T\const(\ell)\left(\norm{x[z](t)-x[\ze](t)}+\norm{v(t)-\ue(t)}\right)\xdt
    +\const(\varphi)\norm{\Delta x}_{\xLinfty}
   \\ \leq \const(\ell,f,g, \varphi, T, R_v, R_x)\label{eq:upperbound_delta1_c}
\end{multline}
Now, let us upper-bound $\Delta_2$.
\begin{equation}
    \Delta_2 = \epsilon\int_0^T \sum_i\left[\psi \circ g_i(x[\zd](t))-\psi \circ g_i(x[\ze](t)) \right]\xdt = \epsilon\sum_i \int_0^T\int_0^1\psi'\left[g_i(x[\ze](t)) + s \Delta g_i(t)\right]\Delta g_i(t) \xd s \xdt 
\end{equation}
From the mean value theorem, \cref{eq:interiority_state_mixed,eq:upper_bound_state_mixed}, $\exists \theta_t\in[0,1]$ such that
\begin{equation}
    \Delta_2 = \epsilon\sum_i \int_0^T\psi'\circ g_i(x[\ze](t) + \theta_t \Delta g(t))\Delta g(t) \xdt \leq \sum_i 2 K_g \norm{\Delta g}_{\xLinfty}
    \leq \const(f,g,K_g, R_x,T)\label{eq:upperbound_delta2_c}
\end{equation}
Now, let us upper-bound $\Delta_3$ defined in \cref{eq:def_delta_3_c}. To do so, let us introduce the following useful subsets of $[0,T]$
 \begin{align}
    E_1&:=(c(x[\ze], \ue))^{-1}\left((-\infty,-C_\delta]\right)\label{eq:def_set_a1_c}\\
    E_2&:=(c(x[\ze], \ue))^{-1}\left((-C_\delta,0]\right)\label{eq:def_set_a2_c}
\end{align}
Let us decompose $\Delta_3$ as follows $\Delta_3 := \Delta_{3,1} + \Delta_{3,2}$, with
\begin{equation}
    \Delta_{3,i} := \int_{E_i} \psi \circ c(x[\zd](t), \vd(t)) - \psi \circ c(x[\ze](t), \ue(t))\xdt,\;i=1,2
\end{equation}
By convexity of the $\log$ penalty, i.e. $\psi$, we have
\begin{equation}
    \label{eq:upperbound_delta31}
    \Delta_{3,1}\leq\int_{E_1} \psi'(C_\delta) \norm{\Delta c}_{\xLinfty} \xdt \leq 
    \const(T,f,c,\psi,C_\delta)\norm{\zd - \ze}_{\Vad}
    \leq \const(T, f, c, \psi, C_\delta, R_v )
\end{equation}
In addition, we have 
\begin{equation}
     \Delta_{3,2} =
     \int_{E_2}\left(\int_{0}^{1}\psi'(c(x[\ze](t), \ue(t)) + s\Delta c(t))\Delta c(t) \xd s\right) \xdt
\end{equation}
Since $\forall t\in E_2, \,\Delta c(t)  < -C_\delta$, we also have
\begin{equation}
    \label{eq:majoS2_c}
    \Delta_{3,2} \leq -C_\delta\int_{E_2}\left(\int_{0}^{1}\psi'(c(x[\ze](t), \ue(t)) + s\Delta c(t))\xd s\right) \xdt
\end{equation}
From the mean value theorem, $\forall t\in E_2$, $\exists \sigma_t\in[0,1]$ such that
\begin{equation}
\label{eq:mean_value_c}
    \psi'\left(c(x[\ze](t), \ue(t)) + \sigma_t \Delta c(t)\right)= \int_{0}^{1} \psi'(c(x[\ze](t), \ue(t)) + s\Delta c(t))\xd s
\end{equation}
Since for all $t\in E_2$, we have $c(x[\ze)(t), \ue(t)) - c(x[\zd](t), \vd(t))\geq C_\delta$ and since $\psi'$ is strictly increasing we have $\sigma_t \in (0,1)$ and 
\begin{equation}
    \psi'\circ c(x[\zd](t), \vd(t)) < \psi'(c(x[\ze](t), \ue(t)) + \sigma_t \Delta c(t)) <\psi'\circ c(x[\ze](t), \ue(t))
\end{equation}
From the intermediate value theorem, $\exists \bar{\sigma}\in(0,1)$ such that $\forall t\in E_2$ we have
\begin{equation}
\label{eq:lower_bound_intermediate_c}
    \psi'(c(x[\ze](t), \ue(t)) + \sigma_t \Delta c(t)) \geq (1 - \bar{\sigma})\psi'\circ c(x[\ze](t), \ue(t)) + \bar{\sigma} \psi'\circ c(x[\zd](t), \vd(t))
\end{equation}
Gathering \cref{eq:majoS2_c,eq:mean_value_c,eq:lower_bound_intermediate_c} yields
\begin{align}
    \Delta_{3,2} &\leq
    - C_\delta \int_{E_2}\bigg( (1-\bar{\sigma})\psi' \circ c(x[\ze](t), \ue(t)) + \bar{\sigma} \psi' \circ c(x[\zd](t), \vd(t)) \bigg)\xdt\\
    &\leq - C_\delta\left( (1 - \bar{\sigma})\int_{E_2}\psi' \circ c(x[\ze](t), \ue(t))\xdt + \const(\bar{\sigma}, \psi, T,c, C_\delta)\right)\label{eq:upperbound_delta32}
\end{align}
Gathering \cref{eq:upperbound_delta31,eq:upperbound_delta32} we have 
\begin{equation}
    \label{eq:upperbound_delta3_c}
    \epsilon \Delta_3 \leq \const(T,f,c,C_\delta,\epsilon_0, \psi, \bar{\sigma},R_v) -\epsilon C_\delta(1 - \bar{\sigma})\int_{E_2}\psi' \circ c(x[\ze](t), \ue(t))\xdt
\end{equation}
Gathering \cref{eq:upperbound_delta1_c,eq:upperbound_delta2_c,eq:upperbound_delta3_c} yields
\begin{multline}
    J_\epsilon(\zd) - J_\epsilon(\ze) = \Delta_1 + \Delta_2 + \epsilon \Delta_3  \leq\\
    \const(\ell,f,g, \varphi, T, K_g,R_v, R_x, C_\delta,c,\epsilon_0, \psi, \bar{\sigma})
    -\epsilon C_\delta(1 - \bar{\sigma})\int_{E_2}\psi' \circ c(x[\ze](t), \ue(t))\xdt
    \label{eq:upperbound_J2}
\end{multline}
Now let us prove \cref{eq:bounded_mixed_penalty_derivative} by contradiction and assume that
\begin{equation}
    \forall K_c>0, \exists \epsilon>0 \st \norm{\epsilon \psi'\circ c(x[\ze], \ue) }_{\xLone}>K_c
\end{equation}
From the definition of $E_1$ and $E_2$, we have
\begin{equation}
    \norm{\epsilon \psi'(c(x[\ze], \ue)}_{\xLone} = 
        \int_{E_1}\epsilon\psi'(c(x[\ze](t), \ue(t))\xdt +
        \int_{ E_2}\epsilon \psi'(c(x[\ze](t), \ue(t))\xdt
\end{equation}
which, in turns yields
\begin{equation}
    \int_{ E_2}\epsilon \psi'(c(x[\ze](t), \ue(t))\xdt > K_c - \epsilon_0\psi'(C_\delta)T
    \label{eq:eq_majo_int_e2_rho_c}
\end{equation}
gathering \cref{eq:upperbound_J2,eq:eq_majo_int_e2_rho_c} yields
\begin{equation}
    J_\epsilon(\zd) - J_\epsilon(\ze) \leq\\
    \const(\ell,f,g, \varphi, T, K_g,R_v, R_x,C_\delta,c,\epsilon_0, \psi, \bar{\sigma})- C_\delta(1 - \bar{\sigma})K_c
\end{equation}
For $K_c$ large enough, $J_\epsilon(\zd) - J_\epsilon(\ze)<0$, which contradicts the optimality of $(x[\ze], \ue)$, proves \cref{eq:bounded_mixed_penalty_derivative} and concludes the proof.
\end{proof}

In addition, using \cref{thm:interior_state} and \cref{thm:interior_mixed_const}, one can also prove a uniform boundedness property for the adjoint state $\bpe$ from \cref{def:penalized_pontryagin_extremal}.
\begin{crllr}
\label{prop:well_posed_p}
Let $(\bue, \bxe)$ be a locally optimal solution of Problem \cref{eq:def_log_barrier_ocp} and let $(\bpe, \bar{\lambda}_\epsilon)$ be the corresponding constraint multipliers, then there exists $K_p<\infty$ such that $\norm{\bpe}_{\xLinfty}\leq K_p$
\end{crllr}

\begin{proof}
First, using \cref{eq:penalized_pontryagin_extremal_2} one has
\begin{multline}
    \norm{\bpe(T)- \bpe(s)} 
    \leq \int_s^T\left(\norm{\ell_x'(\bxe,\bue)}_{\xLinfty}+\norm{f_x'(\bxe,\bue)}_{\xLinfty}\norm{\bpe(t)}\right)\xdt + \sum_i\norm{g_i'(\bxe)}_{\xLinfty}\norm{\epsilon \psi'\circ g_i(\bxe)}_{\xLone}\\
    +\sum_i\norm{c_{i,x}'(\bxe, \bue)}_{\xLinfty}\norm{\epsilon \psi'\circ c_i(\bxe,\bue)}_{\xLone}
\end{multline}
From the continuity of $\ell_x'$, $f_x'$, $g_i'$, $c'_{i,x}$ and since $\bxe$ and $\bue$ are bounded, we have $\norm{\ell_x'(\bxe,\bue)}_{\xLinfty}<\const(\ell,f)$ and  $\norm{f_x'(\bxe,\bue)}_{\xLinfty}<\const(f)$. In addition, the terms on the right-hand side of \cref{eq:penalized_pontryagin_extremal_6} are bounded which yields that $\norm{\bpe(T)}\leq\const(f,h)$. The derivatives of the penalty functions being uniformly $\xLone$-bounded one can use Grönwall Lemma which proves that $\exists K_p<+\infty$ such that $\forall s\in[0,T]$ we have $\norm{\bpe(s)}\leq K_p$ which concludes the proof.
\end{proof}

\begin{lmm}
\label{lem:boundedness_ci}
    There exists a constant $K_c < +\infty$ such that for all $\epsilon>0$, any $(\bxe, \bue)$ locally optimal solution of Problem \cref{eq:def_log_barrier_ocp} satisfies
    \begin{equation}
        c(\bxe(t), \bue(t))\leq - \epsilon/K_c,\textrm{ a.e. }
    \end{equation}
\end{lmm}

\begin{proof}
    For all $K_c>0$, assume that there exists $E\subseteq[0,T]$ of strictly positive measure such that $I^c_{\bue,\bxe(0)}(t,K_c/\epsilon )\neq\emptyset$ for all $t\in E$. Now, let us denote $C(t):= a_{I^c_{\bue,\bxe(0)}(t,K_c/\epsilon )}(\bxe)$ and let us define $v$ as follows
     \begin{equation}
        v(t) := \begin{cases}
            - C(t)^\top\left[C(t) C(t)^\top\right]^{-1}.e_v,\forall t\in E\\
            0\textrm{ otherwise}
        \end{cases}
    \end{equation}
   where $\R^{\vert I^c_{\bue,\bxe(0)}(t,K_c/\epsilon )\vert} \ni e_v :=\begin{pmatrix}\epsilon/K_c &\dots & \epsilon/K_c\end{pmatrix}^\top$. Since $\bue$ is a locally optimal solution, we have for almost all $t\in E$
    \begin{multline}
        H^\psi(\bxe(t), \bue(t)+ v(t), \bpe(t), \epsilon) - H^\psi(\bxe(t), \bue(t) , \bpe(t), \epsilon) = 
        H(\bxe(t), \bue(t)+ v(t),\bpe(t)) - H(\bxe(t), \bue(t) ,\bpe(t))\\
        + \epsilon\sum_{i\in I^c_{\bue,\bxe(0)}(t,K_c/\epsilon)} \log\left(\frac{c_i(\bxe(t), \bue(t))}{c_i(\bxe(t), \bue(t) + v(t))}\right)
    \end{multline}
    From the mean value theorem and the definition of $v$, $\exists s\in[0,1]$ such that 
    \begin{multline}
        H^\psi(\bxe(t), \bue(t)+ v(t), \bpe(t), \epsilon) - H^\psi(\bxe(t), \bue(t) , \bpe(t), \epsilon) \leq \\
        H'_{u}(\bxe(t), \bue(t) + s v(t),\bpe(t)). v(t)  -\epsilon\sum_{i\in I^c_{\bue,\bxe(0)}(t,K_c/\epsilon)} \log(2)
    \end{multline}
    Since $\ell, f$ are at least $\xCtwo$ and  from \cref{prop:well_posed_p} we have
    \begin{multline}
     H^\psi(\bxe(t), \bue(t)+ v(t), \bpe(t), \epsilon) - H^\psi(\bxe(t), \bue(t) , \bpe(t), \epsilon) \leq\\
     \epsilon\left(\frac{\const(\ell, f, K_p, T, R_u)}{K_c}-\sum_{i\in I^c_{\bue,\bxe(0)}(t,K_c/\epsilon)} \log(2) \right)
    \end{multline}
    which is negative for $K_c$ large enough and contradicts the local optimality of $\bue$ and proves the result.
\end{proof}

\section{Convergence of interior point methods in optimal control with logarithmic penalty functions}
\label{sec:convergence}
\subsection{Convergence of state variable and initial-final conditions}
\label{sec:proof_control_state_convergence}
Let us denote $(\bxen, \buen)_n$ a sequence of locally optimal solutions of \cref{eq:def_log_barrier_ocp}. The associated sequence $(\buen, \bxen(0))_n$ being $\xLtwo \times\R^{n}$-bounded, it contains a weakly converging subsequence $(\buenk,\bx_{\epsilon_{n_k}}(0))_k$ satisfying i.e.
\begin{align}
    \lim_{k \rightarrow +\infty} \buenk &\rightharpoonup \bu \\
    \lim_{k \rightarrow +\infty}\norm{\bx_{\epsilon_{n_k}}(0) - \bx^0}&=0
\end{align}
From \cref{lemme:weak_strong_l2_conv}, we also have 
\begin{align}
    \lim_{k\rightarrow +\infty} \norm{x[\buenk, \bx_{\epsilon_{n_k}}(0)] - x[\bu, \bx^0]}_{\xLinfty} &= 0\\
    \lim_{k\rightarrow +\infty}\norm{h\left(x[\buenk, \bx_{\epsilon_{n_k}}(0)](0), x[\buenk, \bx_{\epsilon_{n_k}}(0)](T)\right) - h\left(x[\bu, \bx^0](0),x[\bu, \bx^0])(T)\right)}&=0
\end{align}
which proves \cref{eq:conv_x_Linfty,eq:conv_if}.

\subsection{Convergence of initial-final constraints multipliers}

Let $(\lambda_{\epsilon_n})_n$ be the sequence of multipliers associated with the initial-final constraints \cref{eq:penalized_pontryagin_extremal_4}. These multipliers being bounded there exists a converging subsequence to some $\bar{\lambda}$, which writes $\lim_{k\rightarrow + \infty}\norm{ \bar{\lambda}_{\epsilon_{n_k}} - \bar{\lambda}} = 0$.

\subsection{Convergence of state penalties}
\label{sec:proof_convergence_state_penalty_derivatives}
In this paragraph, we prove that the derivative of the state-constraint penalty converges to a Radon measure $\bar{\mu}\in\mathcal{M}([0,T])^{n_g}$. To do so, let us denote
\begin{equation}
    \xLone([0,T];\R^+)\ni\theta_{\epsilon_n}^{g_i} :=\epsilon_n \psi'\circ g_i(\bxen) = -\frac{\epsilon_n}{g_i(\bxen)}\label{eq:def_lambda_g_mu}
\end{equation}
Identifying any element of the sequence $(\theta_{\epsilon_n}^{g_i})_n$ with a linear form on continuous functions $\Theta_{\epsilon_n}^{g_i}\in \mathcal{M}([0,T])$ defined as follows $\Theta_{\epsilon_n}^{g_i}:v\in \xCzero([0,T],\R) \mapsto \int_0^Tv(t)\theta^{g_i}_{\epsilon_n}(t)\xdt$. From \cref{thm:interior_state,eq:boundedness_lambda} we have $\vert \Theta_{\epsilon_{n_k}}^{g_i}(v)\vert \leq K_g\norm{v}_{\xLinfty} $, thus 
$$\forall \epsilon_n, \Theta_{\epsilon_n}^{g_i}\in B_{\mathcal{M}([0,T])}(0,K_g)$$ From the weak $\ast$ compactness of the unit ball of $\mathcal{M}([0,T])$ (see \cite[Theorem 3.16]{brezis}), there exists a subsequence $\left(\Theta_{\epsilon_{n_k}}^{g_i}\right)_{k\in\mathbb{N}}$ and a measure $\bar{\mu}_i\in \mathcal{M}([0,T])$ with $\bar{\mu}_i(T)=0$ such that $\lim_{k\rightarrow +\infty}\Theta_{\epsilon_{n_k}}^{g_i}\stackrel{\ast}{\rightharpoonup}\bar{\mu}_i,\;\;i=1,\dots,n_g$ which proves \cref{eq:weak_star_conv_pen_state}. Now, let us prove that $\bar{\mu}$ satisfies conditions \cref{eq:first_order_7,eq:first_order_11}. From \cref{thm:interior_state,eq:interiority_state} and from \cref{eq:def_lambda_g_mu}, we have $\theta_{\epsilon_{n_k}}^{g_i}>0$, $\forall t\in[0,T]$ and $\forall \epsilon_{n_k}>0$. Therefore, $\forall \phi\in \xCzero([0,T];\R^+)$ one has $\int \phi \xd \bar{\mu}_i = \lim_k \int\phi \theta_{\epsilon_{n_k}}^{g_i}\xdt\geq0$, which proves \cref{eq:first_order_11}. Finally, let us prove that $\bar{\mu}$ satisfies the complementarity condition \cref{eq:first_order_7}. From \cref{eq:def_lambda_g_mu}, we have $g_i(\bxenk(t))\theta_{\epsilon_{n_k}}^{g_i}(t) = - \epsilon_{n_k}$ hence
\begin{equation}
\label{eq:conv_gin_lambdan_to_zero}
    \lim_{k\rightarrow+\infty}  \int_0^Tg_i(\bxenk(t))\theta_{\epsilon_{n_k}}^{g_i}(t)\xdt = \lim_{k\rightarrow+\infty}-\epsilon_{n_k} T=0
\end{equation}
From the continuity of $g_i$, the sequence $(g_i(\bxenk))_k$ uniformly converges to $g_i(\bx)$. In addition, from \cref{thm:interior_state,eq:boundedness_lambda}, the sequence $\theta_{\epsilon_{n_k}}^{g_i}(t)$ is uniformly $\xLone$-bounded, hence
\begin{equation}
\label{eq:conv_gistar_gicsue}
   \lim_{k\rightarrow +\infty}\left\vert\int_0^T\left(g_i(\bx(t))-g_i(\bxenk(t))\right)\theta_{\epsilon_{n_k}}^{g_i}(t) \xdt \right\vert \leq
   \lim_{k\rightarrow +\infty}\norm{g_i(\bx)-g_i(\bxenk)}_{\xLinfty}\norm{\theta_{\epsilon_{n_k}}^{g_i}}_{\xLone}=0
\end{equation}
Gathering \cref{eq:conv_gin_lambdan_to_zero,eq:conv_gistar_gicsue} yields
\begin{equation}
    \lim_{k\rightarrow +\infty}\int_0^Tg_i(\bx(t))\theta_{\epsilon_{n_k}}^{g_i}(t)\xdt=  \lim_{k\rightarrow +\infty}\int_0^Tg_i(\bxenk(t))\theta_{\epsilon_{n_k}}^{g_i}(t)\xdt = 0
\end{equation}
which in turns gives
\begin{equation}
    \int_0^T g_i(\bx(t))\xd \bar{\mu}_i(t)  = \lim_{k\rightarrow +\infty}\int_0^Tg_i(\bx(t))\theta_{\epsilon_{n_k}}^{g_i}(t) \xdt
     =0
\end{equation}
and proves that $\bar{\mu}$ satisfies the complementarity condition \cref{eq:first_order_7}.
\subsection{Convergence of mixed-constraint penalties}
\label{sec:proof_convergence_mixed_penalty_derivatives}
In this paragraph, we prove that the derivative of the mixed-constraint penalty converges to an assentialyy bounded function $\bar{\nu}\in\xLinfty([0,T];\R_+^{n_g})$. To do so, let us denote
\begin{equation}
     \eta_{\epsilon_n,i} := \epsilon_n \psi'\circ c_i(\bxen, \buen)\label{eq:def_eta_c_nu}
\end{equation}
\begin{prpstn}
\label{prop:mixed_pen_in_Linfty}
 Let $(\bxe, \bue)$ be a locally optimal solution of \cref{eq:def_log_barrier_ocp}, then the following holds
 \begin{equation}
     \eta_{\epsilon_n} \in \xLinfty([0,T];\R^{n_c}_+)
 \end{equation}
 where $\eta_{\epsilon_n,i}:=\begin{pmatrix}\eta_{\epsilon_n,1} &\dots & \eta_{\epsilon_n,n_c} \end{pmatrix}^\top$
\end{prpstn}

\begin{proof}
The proof of this result consists in proving that the mapping 
\begin{equation}
    \Psi_\epsilon: \xLone\ni w \mapsto \int_0^T \eta_{\epsilon_n}(t) w(t)\xdt \in \R
\end{equation}
is a continuous linear form on $\xLone$. From \cref{lem:boundedness_ci}, $c(\bxe, \bue)$ is strictly negative (not active) almost everywhere. Therefore the Hamiltonian minimization condition of the Pontryagin maximum principle writes $H_u'(\bxe,\bue, \bpe) +\sum_{i=1}^{n_c}\eta_{\epsilon,i} a_{i}(\bxe) = 0$ for almost all time. Then, for all $ v\in \U $, one has
    \begin{equation}
        \left\vert \int_0^T \sum_{i=1}^{n_c} \eta_{\epsilon_n,i}(t) a_{i}(\bxe(t)).v(t) \xdt \right\vert \leq
        \norm{ H_u'(\bxe,\bue, \bpe)}_{\xLinfty}\norm{v}_{\xLone} \leq \const(f,\ell,h)\norm{v}_{\xLone}
        \label{eq:Linfty_L1_pen_mixed}
    \end{equation}
    Let us denote $C(t):= a_{I^c_{\bue, \bx^0_{\epsilon}}(t,n)}(\bxe)$ and for all $w\in \xLone([0,T];\R^{n_c})$, let us define $v\in \U$ as follows
    \begin{equation}
        v(t) :=\begin{cases}
                C(t)^\top\left[C(t) C(t)^\top\right]^{-1}w_{I^c_{\bue, \bx^0_ {\epsilon}}(t,n)}(t) &\textrm{if } I^c_{\bue, \bx^0_ {\epsilon}}(t,n)\neq \emptyset\\
                0 &\textrm{otherwise}
        \end{cases}
        \label{eq:def_vt_from_wt}
    \end{equation}
    Since $C(t)$ is $\xLinfty$-bounded, there exists $M>0$ such that $\norm{v}_{\xLone}\leq M \norm{w}_{\xLone}$.
    In addition, let us define $\lambda_\epsilon\in \xLone([0,T];\R^{n_c})$ as follows
    \begin{equation}
        \lambda_{\epsilon,i}(t):= \begin{cases}
            \eta_{\epsilon_n,i} &\textrm{if } i\in I^c_{\bue, \bx^0_{\epsilon}}(t,n)\\
            0 &\textrm{otherwise}
        \end{cases}
        \label{eq:def_lambda_epsilon_t}
    \end{equation}
    Gathering \cref{eq:Linfty_L1_pen_mixed,eq:def_vt_from_wt,eq:def_lambda_epsilon_t} yields
     \begin{equation}
       \left\vert \int_0^T \sum_{i=1}^{n_c} \eta_{\epsilon_n,i}(t) a_{i}(\bxe(t)).v(t) \xdt \right\vert = \left\vert \int_0^T \lambda_\epsilon(t).w(t) \xdt \right\vert
       \label{eq:left_right_equal_lambda_w_mixed}
    \end{equation}
    Gathering \cref{eq:Linfty_L1_pen_mixed,eq:left_right_equal_lambda_w_mixed} we have $\left\vert \int_0^T \lambda_\epsilon(t).w(t) \xdt \right\vert\leq \const(f,\ell,h,M)\norm{w}_{\xLone}$ and using the density of $\xLinfty([0,T];\R^{n_c})$ in $\xLone([0,T];\R^{n_c})$ proves that $\Psi_\epsilon$ is continuous linear form over $\xLone([0,T];\R^{n_c})$ and concludes the proof.
\end{proof}
\textcolor{black}{From \cref{lem:boundedness_ci,prop:mixed_pen_in_Linfty}, $\eta_{\epsilon_n,i}\in B_{\xLinfty}(0,K_c)$, thus there exists a subsequence and a function $\bar{\nu}_i \in \xLinfty([0,T];\R_+)$ such that 
\begin{equation}
    \label{eq:conv_lambda_ci_nu_Lone}
    \lim_{k \rightarrow + \infty} \eta_{\epsilon_{n_k},i} \stackrel{*}{\rightharpoonup} \bar{\nu}_i
\end{equation}
which proves \cref{eq:conv_weak_nu}. Now, let us prove that $\bar{\nu}$ satisfies conditions \cref{eq:first_order_8,eq:first_order_12}. From \cref{thm:interior_mixed_const} and \cref{lem:boundedness_ci}, we have $\eta_{\epsilon_{n_k},i}>0$, $\forall t\in[0,T]$ and $\forall \epsilon_{n_k}>0$ which proves that $\bar{\nu}$ satisfies the non negativity condition \cref{eq:first_order_12}. Finally, let us prove that $\bar{\nu}$ satisfies the complementarity condition \cref{eq:first_order_8}. First, we have
\begin{equation}
\label{eq:conv_cin_etan_to_zero}
    \lim_{k\rightarrow \infty}  \left\langle \eta_{\epsilon_{n_k},i}, c_i(\bxenk, \buenk)\right\rangle = \lim_{k\rightarrow \infty}-\epsilon_{n_k} T=0
\end{equation}
Using \cref{lemme:weak_strong_l2_conv}, yields
\begin{equation}
    \lim_{s\rightarrow \infty}\lim_{r\rightarrow \infty}\left\langle \eta_{\epsilon_{n_s},i},c_i(\bar{x}_{\epsilon_r}, \bar{u}_{\epsilon_r})\right\rangle  = \lim_{r\rightarrow \infty}\lim_{s\rightarrow \infty}\left\langle\eta_{\epsilon_{n_s},i},c_i(\bar{x}_{\epsilon_r}, \bar{u}_{\epsilon_r})\right\rangle = \left\langle\bar{\nu}_i,c_i(\bx, \bu)\right\rangle
\end{equation}
which, in turn, gives
\begin{multline}
    \left\langle\bar{\nu}_i,c_i(\bx, \bu)\right\rangle=\liminf_s \liminf_r  \left\langle \eta_{\epsilon_{n_s},i},c_i(\bar{x}_{\epsilon_r}, \bar{u}_{\epsilon_r})\right\rangle  \\ \leq  \lim_{k\rightarrow \infty}  \left\langle \eta_{\epsilon_{n_k},i}, c_i(\bxenk, \buenk)\right\rangle
    =0 \\ \leq \limsup_s \limsup_r  \left\langle \eta_{\epsilon_{n_s},i},c_i(\bar{x}_{\epsilon_r}, \bar{u}_{\epsilon_r})\right\rangle = \left\langle\bar{\nu}_i,c_i(\bx, \bu)\right\rangle
\end{multline}
which proves that $\bar{\nu}$ satisfies the complementarity condition \cref{eq:first_order_8}.}

\subsection{Convergence of Pontryagin adjoint}
\label{sec:proof_convergence_pontryagin_adjoint}
Let $q_n\in \xBV([0,T])^n $ be the solution of 
\begin{subequations}
\begin{align}
    - \xd q_n(t)&= \begin{multlined}[t]\left[\ell_x'(\bxen(t),\buen(t))+f_x'(\bxen(t),\buen(t)).q_n(t)\right]\xdt\\
    +\sum_{i=1}^{n_c}c_{x,i}'(\bxen(t),\buen(t))\bar{\nu}_i(t)\xdt+\sum_{i=1}^{n_g}g_i'(\bxen(t))\xd \bar{\mu}_i(t)\end{multlined}\\
 q_n(0)&=-h'_{x(0)}(\bxen(0), \bxen(T))^\top.\bar{\lambda} \\
 q_n(T)&=\varphi'(\bxen(T))+h'_{x(T)}(\bxen(0), \bxen(T)) ^\top .\bar{\lambda}
\end{align}
\end{subequations}
Then, we have
\begin{multline}
    q_n(T) - \bp(T) + \bp(t) - q_n(t) = \int^T_t \ell'_x(\bx(s), \bu(s)) - \ell'_x(\bxen(s), \buen(s)) \xd s \\
    +\int^T_t f'_x(\bx(s), \bu(s)) .\bp(s) - f'_x(\bxen(s), \buen(s)).q_n(s) \xd s\\
    + \sum_{i=1}^{n_c} \int^T_t\left[ c_{i,x}'(\bx(s), \bu(s)) - c'_{i,x}(\bxen(s), \buen(s))\right] \bar{\nu}_i(s)\xd s \\
    + \sum_{i=1}^{n_g} \int^T_t\left[ g_{i}'(\bx(s)) - g'_{i}(\bxen(s))\right] \xd \bar{\mu}_i(s)
\end{multline}
Using \cref{lemme:weak_strong_l2_conv}, we have $\lim_n\norm{q_n(T) - \bp(T)} = 0$ and
\begin{align}
    \lim_{n} \norm{\bp(t) - q_n(t)} &=\lim_{n}\norm{ \int^T_t f'_x(\bx(s), \bu(s)) .\bp(s) - f'_x(\bxen(s), \buen(s)).q_n(s) \xd s}\\
    &=\begin{multlined}[t]
        \lim_{n} \norm{\int^T_t f'_x(\bxen(s), \buen(s)) .(\bp(s) - q_n(s)) \xd s}\\
        +\lim_n \norm{\int^T_t \left[f'_x(\bx(s), \bu(s)) - f'_x(\bxen(s), \buen(s))\right] .\bp(s) \xd s}
    \end{multlined}\\
    &\leq \const(f) \int^T_t \norm{\bp(s) - q_n(s)}\xd s
\end{align}
Therefore, $q_n$ pointwise converges to $\bp$ and since both are bounded we have
\begin{equation}
    \lim_{n\rightarrow +\infty} \norm{q_n-\bar{p}}_{\xLone} = 0
    \label{eq:conv_pq_BV}
\end{equation} 
In addition, using \cref{eq:def_lambda_g_mu}, we have
\begin{align}
    \norm{q_n(t)-\bpen(t)}\leq&\begin{multlined}[t] \Bigg{\|}\int^T_t f_x'(\bxen(s),\buen(s)).(q_n(s)-\bpen(s))\xd s\\ + \sum_{i=1}^{n_g}\int_t^Tg_i'(\bxen(s))\left(\xd \bar{\mu}_i(s) - \theta^{g_i}_{\epsilon_n}(s)\xd s\right)\\
    + \sum_{i=1}^{n_c}\int_t^Tc_{i,x}'(\bxen(s), \buen(s))\left(\bar{\nu}_i(s)- \eta_{\epsilon_{n},i}(s)\right)\xd s\Bigg{\|}
    \end{multlined}\\
    \leq&\begin{multlined}[t]\const(f)\int^T_t \norm{q_n(s)-\bpen(s)}\xd s\\ + \sum_{i=1}^{n_g}\norm{\int_t^Tg_i'(\bxen(s))\left(\xd \bar{\mu}_i(s)-\theta^{g_i}_{\epsilon_n}(s)\xd s\right)}\\
    +\sum_{i=1}^{n_c}\norm{\int_t^Tc_{i,x}'(\bxen(s), \buen(s))\left[\bar{\nu}_i(s)-\eta_{\epsilon_{n},i}(s)\right]\xd s}
    \end{multlined}
\end{align}
Now, let us define $h_n\in \xLone([0,T];\R_+)$ as follows
\begin{equation}
h_n(t):=\sum_{i=1}^{n_g}\norm{\int_t^Tg_i'(\bxen(s))\left(\xd\mu_i(s)-\theta^{g_i}_{\epsilon_n}(s)\xd s\right)}
    +\sum_{i=1}^{n_c}\norm{\int_t^Tc_{i,x}'(\bxen(s), \buen(s))\left[\bar{\nu}_i(s)-\eta_{\epsilon_{n},i}(s)\right]\xd s}
\end{equation}
thus $\norm{q_n(t)-\bpen(t)}\leq\const(f)\int^T_t \norm{q_n(s)-\bpen(s)}\xd s + h_n(t)$. From Grönwall inequality \cite[Lemma A.1, p.651]{Khalil}, we have $\norm{q_n(t)-\bpen(t)}\leq \const(f,T)\int^T_th_n(s)\xd s$. From the $\xLinfty$-convergence of $\bxenk$ and the weak $*$ convergence of $\theta_{\epsilon_n}^{g_i}$ we have 
\begin{equation}
\lim_{n\rightarrow \infty} \int_t^Tg_i'(\bxen(s))\left(\xd\mu_i(s)-\theta^{g_i}_{\epsilon_n}(s)\xd s\right) = 0
\end{equation}
\textcolor{black}{From \cref{lemme:weak_strong_l2_conv,eq:conv_lambda_ci_nu_Lone}, we have
\begin{equation}
    \lim_{n\rightarrow \infty}\lim_{m\rightarrow \infty}\left\langle \eta_{\epsilon_{n},i}-\bar{\nu}_i,c'_{i,x}(\bar{x}_{\epsilon_m}, \bar{u}_{\epsilon_m})\right\rangle  = \lim_{m\rightarrow \infty}\lim_{n\rightarrow \infty}\left\langle\eta_{\epsilon_{n},i}-\bar{\nu}_i,c'_{i,x}(\bar{x}_{\epsilon_m}, \bar{u}_{\epsilon_m})\right\rangle = 0
\end{equation}
which, in turn, gives
\begin{multline}
    0=\liminf_n \liminf_m  \left\langle \eta_{\epsilon_{n},i}-\bar{\nu}_i,c'_{i,x}(\bar{x}_{\epsilon_m}, \bar{u}_{\epsilon_m})\right\rangle  \leq  \lim_{n\rightarrow \infty}  \left\langle \eta_{\epsilon_{n},i}-\bar{\nu}_i, c'_{i,x}(\bxen, \buen)\right\rangle
     \\ \leq \limsup_n \limsup_m  \left\langle \eta_{\epsilon_{n},i}-\bar{\nu}_i,c'_{i,x}(\bar{x}_{\epsilon_m}, \bar{u}_{\epsilon_m})\right\rangle =0
\end{multline}
which yields $\lim_{n\rightarrow \infty}\int^T_t c_{i,x}'(\bxen(s), \buen(s))\left[\bar{\nu}_i(s)-\eta_{\epsilon_{n},i}(s)\right]\xd s = 0$. Thus, $h_n$ pointwise convergences to $0$}. In addition, from the boundedness of $(h_n)_n$ there exists a subsequence such that
\begin{equation}
  \lim_{k\rightarrow +\infty}\norm{q_{n_k}(t)-\bpenk(t)}\leq \const(f,T)\int^T_t\lim_{k\rightarrow +\infty}h_{n_k}(s)\xd s=0
\end{equation}
$q_{n_k}$ pointwise converges to $\bpenk$ and since both are bounded, from Lebesgue-Vitali's Theorem, we have
$$\lim_{k\rightarrow +\infty}\norm{q_{n_k}-\bpenk}_{\xLone} = 0$$ 
Gathering with \cref{eq:conv_pq_BV} yields $\norm{\bp -\bpenk}_{\xLone} \rightarrow 0$ which proves \cref{eq:conv_p_L1}.
\subsection{Convergence of control variable and cost function}
\label{sec:proof_convergence_control_cost}

 From, the convexity of $\ell(x,u)$ with respect to $u$ and from the strict convexity of the penalty function we have
\begin{equation}
    \forall n>0, \textrm{ for a.e. } t\in[0,T], {H^{\psi}_{uu}}''(\bxen(t), \buen(t), \bpen(t), \epsilon_n) > 0
\end{equation}
From the implicit function theorem \cite[Theorem 9.27, pp. 224-225]{Rudin}, for almost all time, there exists a mapping $\lambda_t$ such that
\begin{equation}
    \buen(t) := \lambda_t(\bxen(t), \bpen(t), \epsilon_n)
\end{equation}
By continuity of $\lambda_t$ and from the strong $\xLinfty$ (resp. $\xLone$) convergence of $\bxen$ (resp. $\bpen$), $\buen$ pointwise converges to some $z \in \xBV([0,T])^m$. Now, since the sequence $(\buen)_n$ weakly converges to $\bu$, from Mazur's lemma \cite[lemma 10.19, pp. 350]{Renardy.2004}, there exists a function $N:\mathds{N} \mapsto \mathds{N}$ and a sequence of sets of real positive numbers $\left(\left\{\alpha[n]_k: k=n,\dots,N(n)\right\}\right)_n$ satisfying $\sum_{k=n}^{N(n)}\alpha[n]_k = 1$ and such the sequence $(v_n)_n$ defined as follows
\begin{equation}
    v_n := \sum_{k=n}^{N(n)}\alpha[n]_k u_{\epsilon_k}
\end{equation}
converges in $\xLtwo$-norm to $\bu$. Therefore, there exists a subsequence denoted $(v_m)_m$ converging almost everywhere to $\bu$. Now, for almost all $t\in[0,T]$, we have
\begin{align}
    \norm{\bu(t) - z(t)} = \lim_{m\rightarrow +\infty} \norm{v_m(t) - z(t)} \leq  \lim_{m\rightarrow +\infty} \sum_{k=m}^{N(m)}\alpha[m]_k\norm{\bu_{\epsilon_m}(t) - z(t)} =0
\end{align}
which proves that there exists a subsequence$(\buenk)_k$ which converges almost everywhere to $\bu$ and since $(\buen)_n$ is $\xLinfty$-bounded the subsequence converges in $\xLone$-norm to $\bu$ which proves \cref{eq:conv_u_L2} and \cref{eq:conv_cost} from $\ell\in \xCone$.
\subsection{Convergence of stationary conditions on the Hamiltonian}
\label{sec:proof_convergence_stationnarity}
Let us prove that the limit point of the solution of \cref{eq:penalized_pontryagin_extremal} is also a solution of \cref{eq:first_order_3}. From \cref{eq:conv_u_L2,eq:conv_x_Linfty}
\begin{multline}
    \lim_{n\rightarrow +\infty }\norm{{H_u^{\psi}}'(\bxen, \buen, \bpen, \epsilon_n) - H'_u(\bx, \bu, \bp) - b(\bx)^\top.\bar{\nu}}_{\xLone} \leq \\
    \lim_{n\rightarrow +\infty }\bigg[\norm{\ell'_u(\bxen, \buen)-\ell'_u(\bx, \bu)}_{\xLone} +  \norm{f_2(\bxen)^\top . \bpen -f_2(\bx)^\top . \bp }_{\xLone} \\
    +\sum_{i=1}^{n_c}\norm{\epsilon_n \psi'\circ c_i(\bxen, \buen) b_i(\bxen) - \bar{\nu}_i b_i(\bx)}_{\xLone}\bigg]
\end{multline}
Each term of the sum converges in $\xLone$-norm to 0. Thus, taking a subsequence if necessary, \cref{eq:penalized_pontryagin_extremal_3} converges to $0$ almost everywhere and proves \cref{eq:first_order_3}.

\section{Solving Algorithms}
\label{sec:solving_alg}
\subsection{Primal solving algorithm}
In \cref{thm:first_order_convergence} we have proved that any sequence of solutions of \cref{eq:penalized_pontryagin_extremal} contains a converging subsequence. In the following, we denote $S_P(\epsilon):=(\bxe, \bpe, \bue,\bar{\lambda}_\epsilon)$ any solution of \cref{eq:penalized_pontryagin_extremal}. Now, the primal solving algorithm naturally writes as follows
\begin{algorithm}
\caption{Primal algorithm for optimal control problems}
\label{alg:primal}
\begin{algorithmic}[1]
\STATE{Define $\epsilon_0>0,\; \alpha \in(0,1),\;\textrm{tol}=o(1),\;k=0 $}
\WHILE{$\epsilon_k > \textrm{tol}$}
\STATE{$S_P(\epsilon_{k+1})\gets $solution of \cref{eq:penalized_pontryagin_extremal} initialized with $S_P(\epsilon_k)$} 
\STATE{$\epsilon_{k+1}\gets \alpha\epsilon_k$}
\STATE{$k\gets k+1$}
\ENDWHILE
\RETURN $S_P(\epsilon_k)$
\end{algorithmic}
\end{algorithm}

\subsection{Primal-dual solving algorithm}
\label{sec:primal_dual_conv}
 Before describing the primal-dual solving algorithm, we need the following convergence result, which is a direct consequence of \cref{thm:first_order_convergence}.
\begin{thrm}
\label{thm:primal_dual}
Let $(\epsilon_n)$ be a sequence of decreasing positive parameters with $\epsilon_n\rightarrow0$ and let
\begin{multline}
   (\bxen, \buen, \bpen, \bar{\theta}_{\epsilon_n}, \bar{\eta}_{\epsilon_n},\bar{\lambda}_{\epsilon_n})_n\in\\ \xWn{{1,\infty}}([0,T];\R^n) \times \textnormal{U} \times \xWn{{1,1}}([0,T];\R^n) \times \xLone([0,T];\R^{n_g}_+) \times \xLinfty([0,T];\R^{n_c}_+) \times \R^{n_h}
\end{multline}

be a solution of the following Primal Dual TPBVP
\begin{subequations}
\label{eq:all_primal_dual}
\begin{align}
    \dot{x}_{\epsilon_n}(t) =& f(\bxen(t),\buen(t))\label{eq:primal_dual_1}\\
    \dot{\bar{p}}_{\epsilon_n}(t) =&\begin{multlined}[t]
    -H'_x(\bxen(t), \buen(t), \bpen(t))- \sum_{i=1}^{n_g}\bar{\theta}_{\epsilon_n,i}(t)g_i'(\bxen(t))\\
    - \sum_{i=1}^{n_c}\bar{\eta}_{\epsilon_n,i}(t)c_{i,x}'(\bxen(t), \buen(t))\end{multlined}\label{eq:primal_dual_2}\\
    0=&H'_u(\bxen(t), \buen(t), \bpen(t))+\sum_{i=1}^{n_c}\bar{\eta}_{\epsilon_n,i}(t)c'_{i,u}(\bxen(t),\buen(t))\label{eq:primal_dual_3}\\
    0=&\bar{\theta}_{\epsilon_n,i}(t) - g_i(\bxen(t))-\sqrt{\bar{\theta}_{\epsilon_n,i}(t)^2+g_i(\bxen(t))^2+2\epsilon_n}\label{eq:primal_dual_4}\\
    0=&\bar{\eta}_{\epsilon_n,i}(t) - c_i(\bxen(t), \buen(t))-\sqrt{\bar{\eta}_{\epsilon_n,i}(t)^2+c_i(\bxen(t), \buen(t))^2+2\epsilon_n}\label{eq:primal_dual_5}\\
    0 =&h(\bxen(0),\bxen(T))\label{eq:primal_dual_6}\\
    0=& \bpen(0)  +h'_{x(0)}(\bxen(0), \bxen(T))^\top.\bar{\lambda}_{\epsilon_n}\label{eq:primal_dual_7}\\
    0 =&\bpen(T) -\varphi'(\bxen(T))- h'_{x(T)}(\bxen(0), \bxen(T))^\top.\bar{\lambda}_{\epsilon_n}
\end{align}
\end{subequations}
Then $(\bxen, \buen, \bpen, \bar{\theta}_{\epsilon_n}, \bar{\eta}_{\epsilon_n})_n$ contains a subsequence converging to a stationary point of the original problem $(x[\bu, \bx^0],\bu,\bp,\bar{\mu},\bar{\nu}, \bar{\lambda})$ as follows
\begin{subequations}
\begin{multline}
     \norm{\buenk - \bu}_{\xLone} \rightarrow 0,\;\;\norm{\bxenk -x[\bu, \bx^0]}_{\xLinfty} \rightarrow 0,\;\;\vert J(\bxenk,\buenk) - J(x[\bu, \bx^0], \bu) \vert \rightarrow 0\\
     \norm{\bar{\lambda}_{\epsilon_{n_k}}-\bar{\lambda}} \rightarrow 0,\;\;\norm{\bpenk-\bp}_{\xLone} \rightarrow 0,\;\;\textcolor{black}{\bar{\eta}_{\epsilon_n} \stackrel{\ast}{\rightharpoonup}\bar{\nu}},\;\;\bar{\theta}_{\epsilon_n} \xd t \stackrel{\ast}{\rightharpoonup}
    \xd \bar{\mu} 
\end{multline}
\end{subequations}
\end{thrm}
\begin{proof}
From \cref{thm:interior_state} and \cref{lem:boundedness_ci}, we have $g_i(\bxen(t))<0$ and  $c_i(\bxen(t),\buen(t))<0$ for all $\epsilon_n>0$. Therefore \cref{eq:primal_dual_4} is equivalent to $\bar{\theta}_{\epsilon_n,i}(t) =-\epsilon_n/g_i(\bxen(t))$ and \cref{eq:primal_dual_4} is equivalent to $\bar{\eta}_{\epsilon_n,i}(t) =-\epsilon_n/c_i(\bxen(t), \buen(t))$. Combining with \cref{eq:primal_dual_2,eq:primal_dual_3} proves that any solution of \cref{eq:all_primal_dual} is also solution of \cref{eq:penalized_pontryagin_extremal} and using \cref{thm:first_order_convergence} concludes the proof.
\end{proof}
Let us denote $S_{PD}(\epsilon):=(\bxe, \bue, \bpe, \bar{\theta}_{\epsilon}, \bar{\eta}_{\epsilon},\bar{\lambda}_{\epsilon})$ any solution of \cref{eq:all_primal_dual}, then the primal-dual algorithm writes as follows.
\begin{algorithm}
\caption{Primal-dual algorithm for optimal control problems}
\label{alg:primal_dual}
\begin{algorithmic}[1]
\STATE{Define $\epsilon_0>0,\; \alpha \in(0,1),\;\textrm{tol}=o(1),\;k=0 $}
\WHILE{$\epsilon_k > \textrm{tol}$}
\STATE{$S_{PD}(\epsilon_{k+1})\gets $solution of \cref{eq:all_primal_dual} initialized with $S_{PD}(\epsilon_k)$} 
\STATE{$\epsilon_{k+1}\gets \alpha\epsilon_k$}
\STATE{$k\gets k+1$}
\ENDWHILE
\RETURN $S_{PD}(\epsilon_k)$
\end{algorithmic}
\end{algorithm}

Even though \cref{alg:primal,alg:primal_dual} are equivalent, the primal-dual algorithm can explore non-admissible trajectories without becoming singular. For example, the primal-dual algorithm can be initialized with non-admissible trajectories and still be numerically tractable which is of course, not the case with the primal algorithm.

\section{Numerical example: Robbin's problem}
\label{sec:algoExample}
The numerical example and the Differential Algebraic Equations (DAEs) solver used in this example are freely available at \url{https://ifpen-gitlab.appcollaboratif.fr/detocs/ipm_ocp}. The solver is a two point boundary differential algebraic equations solver  adapted from \cite{Kierzenka2001ABS} to solve index-1 differential algebraic equations.
\begin{subequations}
\begin{align}
    \min_{u} & \int_0^6 x(t)\xdt\\
    x'''(t) & = u(t)\\
    x(0) &= 1\\
    x'(0), x''(0) &= 0\\
    0&\geq -x(t)\\
    u & \in [-1,1]
\end{align}
\end{subequations}
This problem is challenging since the optimal solution exhibits a Fuller-like phenomenon both on the adjoint state $\bp$ and on the control $\bu$. For this problem, one can check that conditions \cref{eq:penalized_pontryagin_extremal_5,eq:penalized_pontryagin_extremal_6} are equivalent to dropping the end initial-final constraint multiplier and add the constraint $\bar{p}(T) = 0$.

\subsection{Resolution using the primal algorithm}
The parameterization of the primal algorithm for the Robbins problem is as follows
\begin{equation}
\epsilon_0=0.1,\;\;\alpha=0.8,\;\;\textrm{tol}=10^{-8},\;\;S_P(\epsilon_0)(t):=(1,0,0,0,0,0,0)
\end{equation}
Using this setting, the execution time is 0.97s.

\subsection{Resolution using the primal-dual algorithm}
The parameterization of the primal-dual algorithm for the Robbins problem is as follows
\begin{equation}
    \epsilon_0=0.1,\;\;\alpha=0.5,\;\;\textrm{tol}=10^{-9},\;\;S_{PD}(\epsilon_0)(t):=(1,0, 0, 0, 0, 0, 0, 0, 0, 0)
\end{equation}
One can see that the decay rate of the primal-dual method is smaller than the one used in the primal case. Both parameters have been set to the lower value achieving convergence. In addition, the tolerance can also be set lower using the primal-dual version of the algorithm. The execution time with the primal-dual method is 0.20s thanks to the smaller decay rate.

\appendix

\section{Proofs of \cref{sec:preliminary_results}}
\label{sec:proof_annex}

\subsection{Proof of \cref{prop:major_Linf_L1}}
\label{sec:proof_lipschitz}
From \cref{ass:bounded_control_set}, $x[u, x^0]$ is valued in a compact subset of $\R^n$, In addition, $f$ being $\xCtwo$ there exists $\const(f) < +\infty$ such that for all $(u_1, x^0_1),(u_2,x^0_2)\in \xLinfty \times \R^n$
\begin{equation}
\parallel \dot{x}[u_1, x^0_1](t) - \dot{x}[u_2, x^0_2](t) \parallel 
\leq \const(f)\left(\parallel x[u_1, x^0_1](t) - x[u_2, x^0_2](t) \parallel  + \parallel u_1(t) - u_2(t)\parallel \right)    
\end{equation}
Using Grönwall inequality \cite[Lemma A.1, p.651]{Khalil} again, there exists $\const(f) < +\infty$ such that $\parallel x[u_1, x^0_1]-x[u_2,x^2_0]\parallel_{\xLinfty} \leq \const(f)(\parallel u_1 - u_2\parallel_{\xLone} +\norm{x^0_1 - x^0_2})$.

\subsection{Proof of \cref{lemme:weak_strong_l2_conv}}
\label{app:weak_strong_convergence}
To alleviate the notation, we denote $x_n:=x[u_n,x^0_n]$ and $\bx:=x[\bu, \bx^0]$. using these notations, we have
\begin{equation}
    x_n(t_2) - x_n(t_1) := \int_{t_1}^{t_2}f_1(x_n(t)) + f_2(x_n(t)).u_n(t)\xdt
\end{equation}
First, $\forall t_1,t_2\in[0,T]$, From Hölder inequality, we have
\begin{equation}
    \norm{x_n(t_2) -  - x_n(t_1)}\leq\sup_n\norm{f_1(x_n) + f_2(x_n).u_n}_{\xLtwo}\sqrt{\vert t_1 - t_2\vert} 
\end{equation}
Therefore, the sequence $(x_n)_n$ is bounded and equicontinuous. From Arzela-Ascoli \cite[Theorem 1.3.8, p.33]{kurdila}, it contains a uniformly converging subsequence to some $\hat{x}$. Let $(x_k)_k$ be the uniformly converging sequence, one has
\begin{equation}
    \hat{x}(t) - \bx(t) = \lim_{k\rightarrow +\infty } x_k(t) -\bx(t)
  =\int_0^t f_1(\hat{x}(s))- f_1(\bx(s)) + (f_2(\hat{x}(s)) -f_2(\bx(s)) ).\bu(s)\xd s \label{eq:diff_xhat_xbar}
\end{equation}
Since $\hat{x}(0)=\bx^0$, the term inside the integral in \cref{eq:diff_xhat_xbar} is always zero, thus $\hat{x} = \bx$. Now, let us prove that the whole sequence $(x_n)_n$ uniformly converges to $\bx$ by contradiction. Assume that, there exists a subsequence $(x_k)_k$ such that $\exists K>0$ and $\epsilon >0$ satisfying $\norm{x_k- \bx}_{\xLinfty}\geq \epsilon$ for all $k\geq K$. One can extract a sub-subsequence $(x_{k_j})_j$ uniformly converging to some $x_1$ with $\norm{x_1 - \bx}_{\xLinfty}>0$. However, $\Vad$ being weakly compact in the topology $\sigma(\xLtwo\times \R^n, \xLtwo\times \R^n)$, one can extract from $(u_{k_j}, x^0_{k_j})_j$ a weakly convergent subsequence. By definition, this sequence weakly converges to $(\bu, \bx^0)$ and proves that $(x_{k_j})_j$ contains a subsequence converging to $\bx$ which contradicts the initial assumption and proves the uniform convergence of $(x_n)_n$. As a consequence, $\alpha(x_n)$ and $\beta(x_n)$ uniformly converges to $\alpha(\bx)$ and $\beta(\bx)$ respectively. In addition, the sequence $(\alpha(x_n).u_n + \beta(x_n))_n$ is uniformly $\xLinfty$-bounded, thus contains a weakly $\ast$ converging subsequence to some $\theta $. Assume that $\theta$ is not equal to $\alpha(\bx).\bu +\beta(\bx)$ and Let $\varphi \in \xLtwo([0,T];\R^{n_c})\cap \xLone([0,T];\R^{n_c})$ we have
\begin{equation}
    \langle\varphi, \theta - \alpha(\bx).\bu -\beta(\bx) \rangle = \lim_{k\rightarrow +\infty}\left\langle \varphi, \alpha(x_{n_k}).u_{n_k} + \beta(x_{n_k}) - \alpha(\bx).\bu - \beta(\bx)\right\rangle
    = \lim_{k\rightarrow +\infty}\left\langle \varphi, \alpha(\bx) (u_{n_k} -\bu) \right\rangle = 0
\end{equation}
Since $\xLtwo([0,T];\R^{n_c})\cap \xLone([0,T];\R^{n_c})$ is dense in $\xLone([0,T];\R^{n_c})$ this contradicts the initial assumption and proves that the weak $*$ limit is $\alpha(\bx).\bu +\beta(\bx)$. To prove that the whole sequence weakly $*$ converges, we use the same argument as the one we used to prove uniform convergence of the state, which concludes the proof.

\subsection{Proof of \cref{prop:lower_bound_measure}}
\label{sec:proof_measure}
Using \cref{ass:bounded_control_set}, $\forall u\in \U$ we have
\begin{equation}
    \vert g_i(x[u](t)) - g_i(x[u](s))\vert \leq \const(g)\norm{x[u](t) - x[u](s)}
    \leq \const(f,g)\vert t-s\vert
\end{equation}
To prove the proposition, we only need to prove the lower bound holds on any interval $(\alpha_1,\alpha_2)\subseteq E$. From the continuity of $g_i$, $\exists t_1,t_2$ such that $g_i(x[u](t_1))=\alpha_1$, $g_i(x[u](t_2))=\alpha_2$ and such that $(t_1,t_2)\subseteq g_i(x[u])^{-1}((\alpha_1,\alpha_2))$ and
\begin{align}
  m[u,g_i]((\alpha_1,\alpha_2)) \geq \vert t_1 - t_2\vert&\geq\const(f,g) \vert g_i(x[u](t_1)) - g_i(x[u](t_2))\vert\nonumber\\
  &\geq \const(f,g)\vert \alpha_1 - \alpha_2\vert  
\end{align}

\subsection{Proof of \cref{prop:maximal_distance}}
\label{sec:proof_maximal_distance}
Let $\delta>0$, and for all $(u, x^0)\in \Vad_\infty$ let us denote
\begin{equation}
    \gamma_\delta(u, x^0) := \inf_{v\in B_{\norm{.}_{\Vad}}((u,x^0),\delta)\cap \Vad}\left\lbrace \sup_t g(x[v, y^0](t))\right\rbrace
\end{equation}
From \cref{ass:interior_accessibility}, we have $\gamma_\delta(u, x^0)<0$. Then, $\forall (u, x^0)\in \Vad_\infty$, $\exists (v,y^0)\in B_{\norm{.}_{\Vad}}((u,x^0),\delta)\cap \Vad$ such that
\begin{equation}
\sup_t g(x[v, y^0](t))\leq \gamma_\delta(u, x^0)\leq \sup_{(u, x^0)\in \Vad_\infty}\gamma_\delta (u, x^0):=-2G_\delta <0
\end{equation}
In addition,  if $\supp^g_{u,x^0}(G_\delta)\neq \emptyset$ , then $\forall t\in \supp^g_{u,x^0}(G_\delta)$ we have
\begin{align}
    g(x[v, y^0](t)) - g(x[u, x^0](t))\leq -2G_\delta + G_\delta =- G_\delta
\end{align}
Now, let us denote  
\begin{equation}
    \kappa_\delta(u,x^0) := \inf_{v\in B_{\norm{.}_{\Vad}}((u,x^0),\delta)\cap \Vad}\left\lbrace \esssup_t c(x[v, y^0](t), v(t))\right\rbrace
\end{equation}
From \cref{ass:interior_accessibility}, for all $(u, x^0)\in \Vad_\infty$, $\exists (v_n, y^0_n)_n\in \VadStrict(n)$ converging to $(u, x^0)$. Thus, for all $\delta >0$, $\exists N_\delta[u, x^0]>0$, such that $\forall n\geq N_\delta[u, x^0]$, we have $(v_n,y_n^0)\in B_{\norm{.}_{\Vad}}((u, x^0), \delta)$, which yields $\kappa_\delta(u,x^0) \leq -1/N_\delta(u, x^0)$. Now, $\forall (u, x^0)\in \Vad_\infty$, $\exists (v,y^0)\in B_{\norm{.}_{\Vad}}((u,x^0),\delta)\cap \Vad$ such that
\begin{equation}
    \esssup_t c(x[v, y^0](t), v(t))\leq \kappa_\delta(u, x^0)\leq\sup_{(u, x^0)\in \Vad_\infty} -\frac{1}{N_\delta[u, x^0]}:=-2C_\delta <0
\end{equation}
In addition,  if $\supp^c_{u,x^0}(C_\delta)\neq \emptyset$ , then $\forall t\in \supp^c_{u, x^0}(C_\delta)$ we have
\begin{align}
    c(x[v, y^0](t), v(t)) - c(x[u, x^0](t), u(t))\leq -2C_\delta + C_\delta =- C_\delta
\end{align}

\bibliographystyle{plain}
\bibliography{references}

\end{document}